\documentclass{amsart}
\usepackage[utf8]{inputenc}
\usepackage{amssymb, amsmath, latexsym, amsthm, amscd, fancyhdr, graphicx, xcolor, hyperref, bm, mathrsfs}
\DeclareMathAlphabet{\mathpzc}{OT1}{pzc}{m}{it}

\usepackage{nicefrac}
\def\bbz{\mathbb{Z}}
\def\bba{\mathbb{A}}

\def\bbn{\mathbb{N}}

\def\frakr{\mathfrak{r}}
\def\fraka{\mathfrak{a}}

\def\frakg{\mathfrak{g}}
\def\gfrak{\mathfrak{g}}
\def\frakh{\mathfrak{h}}

\def\frako{\mathfrak{o}}

\def\bG{\mathbf{G}}
\def\bH{\mathbf{H}}

\def\sl{\operatorname{SL}}
\def\ad{\operatorname{ad}}
\def\Ad{\operatorname{Ad}}
\def\vol{\operatorname{vol}}
\def\sob{\mathcal{S}}
\def\scal{\mathcal{S}}

\def\tF{\tilde F}
\def\tw{\tilde{w}}
\def\tGw{\tilde{G}_{\tilde{w}}}
\def\bu{\mathbf u}

\newtheorem{theorem}{Theorem}[section]

\def\norm#1{\left\Vert #1\right\Vert}

\newtheorem{proposition}[theorem]{Proposition}
\newtheorem{corollary}[theorem]{Corollary}
\newtheorem{lem}[theorem]{Lemma}
\newtheorem{lemma}[theorem]{Lemma}

\newtheorem{remark}[theorem]{Remark}

\newcommand{\levelbound}{L}

\newcommand{\diff}{\operatorname{d}}
\newcommand{\Sob}{\mathcal{S}}

\newcommand{\Lie}{\mathrm{Lie}}

\newcommand{\Siegel}{\Sieg}

\newcommand{\G}{\mathbf{G}}

\newcommand{\Sieg}{\mathfrak{S}}
\newcommand{\height}{\mathrm{ht}}

\newcommand{\m}{\underline{m}}

\newcommand{\SL}{\mathrm{SL}}

\renewcommand{\H}{\mathbf{H}}

\newcommand{\Q}{\mathbb{Q}}

\newcommand{\adele}{\mathbb{A}}

\newcommand{\vare}{\varepsilon}
\newcommand{\al}{a_\lambda}
\newcommand{\ul}{u_\lambda} 
\newcommand{\order}{\mathfrak o}

\newcommand{\Gfrak}{\mathfrak G}

\newcommand{\red}{\operatorname{red}}
\newcommand\disg{\mathsf{g}}

\newcounter{consta}

\newcounter{constk}
\renewcommand{\theconstk}{{\kappa_{\arabic{constk}}}}
\newcommand{\constk}{\refstepcounter{constk}\theconstk}

\newcounter{constc}

\newcounter{constE}
\renewcommand{\theconstE}{{{C}_{\arabic{constE}}}}
\newcommand{\constE}{\refstepcounter{constE}\theconstE}

\makeatletter
\newcommand*\bigcdot{\mathpalette\bigcdot@{.5}}
\newcommand*\bigcdot@[2]{\mathbin{\vcenter{\hbox{\scalebox{#2}{$\m@th#1\bullet$}}}}}
\makeatother

\title{Property $(\tau)$ in positive characteristic}

\author{Amir Mohammadi}
\address{A.M.: Department of Mathematics, University of California, San Diego, CA 92093}
\email{ammohammadi@ucsd.edu}
\thanks{A.M.\ was partially supported by the NSF}

\author{Nattalie Tamam}
\address{Department of Mathematics, University of Michigan, Ann Arbor, Michigan, MI 48109}
\email{nattalie@umich.edu}
\date{}

\begin{document}

\maketitle

\begin{abstract}
    We prove a quantitative equidistribution statement for certain adelic homogeneous subsets in positive characteristic. 
As an application, we describe a proof of property $(\tau)$ for arithmetic groups in this context.
\end{abstract}

\tableofcontents

\section{Introduction}\label{sec:introductio}

Let~$X=L/\Lambda$ be the quotient of a locally compact group $L$ by a lattice~$\Lambda\subset L$. Any subgroup~$M\subset L$ acts on~$X$ by left multiplication. 
A {\em homogeneous measure} on $X$ is, by definition, a probability measure $\mu$ that is supported
on a single closed orbit $Y=M_Yg\Lambda$ of its stabilizer $M_Y=\mathrm{Stab}(\mu)$. A {\em homogeneous set} is the support of a homogeneous measure.  

Motivated by problems in number theory, the limiting behavior of a sequence $\{\mu_i\}$ of homogeneous measures has been extensively studied. When $L$ is a Lie group and ${\rm Stab}(\mu_i)$ is generated by unipotent subgroups, Mozes and Shah~\cite{Mozes-Shah} used Ratner's celebrated measure classification theorems~\cite{Ratner-measure} and the Linearization techniques of Dani and Margulis~\cite{DM-Linearization} to obtain a very satisfactory classification theorem for the limiting measure, see also~\cite{Eskin-Mozes-Shah-1996, GO-Manin, ELMV-1}.  

More recently, the quantitative aspects of this problem have attracted considerable attention. Indeed, in a landmark paper, Einsiedler, Margulis, and Venkatesh~\cite{EMV} proved an effective equidistribution theorem for homogeneous measures when $M$ and $L$ are semisimple Lie groups and $\Lambda$ is a congruence lattice, under some additional conditions; and in a followup work Einsiedler, Margulis, Venkatesh and the first named author~\cite{EMMV} extended this result to certain adelic periods in the number field setting. 

In this paper, we consider homogeneous measures in positive characteristic setting. Let us fix some notation in order to state our main results. Let $F$ be a global function field, and let $\mathbf{G}$ be an absolutely almost simple, simply connected $F$-group. Throughout the paper, we assume $\G$ does {\em not} admit non-standard isogenies. This condition is always satisfied if ${\rm char}(F)>3$. More precisely, $\G$ admits non-standard isogenies only in the following cases: if ${\rm char}(F)=2$ and $\G$ is of type $B_n, C_n, (n\geq 2)$ or of type $F_4$, another case is if ${\rm char}(F)=3$ and $\G$ is of type $G_2$, see~\cite[\S1]{Pink-Compact} and references therein.

Put $\bG'=\bG\times \bG$, and let 
\[
\bH=\{(h,h): h\in\G\}\subset \bG'
\]
be the diagonal embedding of $\bG$ in $\bG'$. 

Throughout the paper, $\Sigma$ denotes the set of all places of $F$. 
For each place $v\in \Sigma$, let $F_v$ be the completion of $F$ at $v$, and let $\bba_F$ (or simply $\bba$ if there is no confusion) denote the ring of adeles over $F$.

Let $G'=\bG'(\bba)$ and $X=G'/\bG'(F)$. 
We let $m_X$ denote the $G'$-invariant probability measure on $X$.
Let $\mathpzc i:\bH\to\bG'$ be the above inclusion. For every $g\in G'$, 
\[
Y_g:=g\mathpzc i\bigl(\bH(\bba)/\bH(F)\bigr)
\]
is a closed orbit in $X$ equipped with a probability measure $\mu_g$ which is invariant under $H_g:=g\bH(\bba)g^{-1}$. Throughout the paper, we drop $\mathpzc i$ from the notation and simply write $Y_g=g\bH(\bba)/\bH(F)$.

We will use the following notion of complexity which was introduced and studied in~\cite{EMMV}. Let $\Omega_0\subset G'$ be a compact open subgroup. 
Given $(Y_g,\mu_g)$ as above, define 
\begin{equation}\label{eq: def volume}
\vol(Y_g)=\mathsf{m}_g(\Omega_0\cap H_g)^{-1}    
\end{equation}
where $\mathsf{m}_g$ denotes a Haar measure on $H_g$ which projects to $\mu_g$. Note that a different choice of $\Omega_0$ only changes the above function by a multiplicative constant, see \cite[\S 2.3]{EMMV} for a proof.  

The following is the main theorem of this paper.

\begin{theorem}\label{thm;equidistribution}\label{thm:main}
    There exists $\kappa>0$ so that the following holds. Let $g\in G'$, and let $(Y_g,\mu_g)$ be as above. Then   
    \[
    \left|\int_{Y_g} f\diff\!\mu_{g}-\int f\diff\!m_X\right|\ll \vol(Y_g)^{-\kappa}\mathcal{S}(f),\qquad\text{$f\in C_c^\infty(X)$},
    \]
    where $\sob(f)$ denotes a certain adelic Sobolev norm, and the implied multiplicative constant depends on $X$. The exponent $\kappa$ depends only on the type of $\G$ and $F$. If $X$ is compact, then $\kappa$ depends only on the type of $\G$. 
\end{theorem}

Similar to the proofs in~\cite{EMV,EMMV}, our proof relies on uniform spectral gap. However, our treatment deviates from loc.\ cit.\ in that instead of using effective ergodic theorems for the action of unipotent subgroups, we use averages over expanding pieces of root subgroups, see \S\ref{sec:generic pts}. This idea (which is due to Margulis, see~\cite{Moh-MarkovSpec}) makes the analysis in the case at hand more transparent. We also use the work of Prasad~\cite{Prasad-VolForm}; albeit our application here is less intricate than in~\cite{EMMV}, in particular, our treatment does not rely on~\cite{BP-Finiteness}, as $\mathbf H\subset \G'$ is simply connected and {\em embedded}.       

\medskip

As it was mentioned our proof relies on uniform spectral gap. However, it also allows us to give an independent proof of property $(\tau)$ in all cases except for groups of type $A_1$. 
That is, if we assume property $(\tau)$ for groups of type $A_1$, Drinfeld~\cite{Drinfeld}, and use the uniform estimates on decay of matrix coefficients for groups with Kazhdan's property $(T)$,  Oh~\cite{HeeOh-T}, we can deduce property $(\tau)$ in all other cases as well as our theorem.  

In the number field setting, the proof of property $(\tau)$ has a long history and was completed thanks to several deep contributions, see~\cite{Kazhdan-T, HeeOh-T, Selberg-ThreeSixteenth, Jacquet-Langlands, Burger-Sarnank, Cl-tau, Gor-Mau-Oh}. 
In the positive characteristic setting, the existing literature is far less satisfactory beyond groups with property $(T)$, the aforementioned work of Drinfeld~\cite{Drinfeld}, and the work of Lafforgue~\cite{Lafforgue1, Lafforgue2}. However, the statement itself is agreed upon by the experts. We hope the following theorem will contribute to filling this lacuna in the literature.

\begin{theorem}[Property $\tau$] \label{thm: property tau}
Let the notation and assumptions be as above, in particular, we assume that $\G$ does {\em not} admit non-standard isogenies. 

	Let $v\in\Sigma$ be such that $\G$ is isotropic over $F_v$. Then $L^2_0(\G(\adele)/\G(F))$, 
	the orthogonal complement of $\G(\adele)$-invariant functions, is isolated from the trivial representation as a representation of $\G(F_v)$. Moreover, this isolation (spectral gap) is independent of the $F$-form of $\G$. 
\end{theorem}

\subsection*{Acknowledgment}
We would like to thank H.\ Oh for her interest in this project and also for several helpful discussions. We are particularly grateful to her for communications regarding~\cite{HeeOh-T}. We would also like to thank G.\ Prasad, I.\ Rapinchuk, and A.\ Salehi Golsefidy for helpful discussions.

\section{Notation and preliminaries}\label{notation}\label{sec:notation}
Let $F$ be a global function field. Let $\Sigma$ be the set of places on $F$, and let $\bba$ be the ring of adeles over $F$. 
For each place $v\in \Sigma$, let $F_v$ be the completion of $F$ at $v$. Let $\order_v$ be the ring of $v$-integers in $F_v$; $k_v$ denotes the residue field of $\order_v$ and $\varpi_v$ is a uniformizer of $\frako_v$. We will denote by $|x|_v$ the absolute value on $F_v$.
Note that $\frako_v$ is the maximal compact subring of $F_v$. Let $B_{v}(a,r)=\{b\in F_v: |a-b|_v\leq r\}$.

As in the introduction, let $\mathbf{G}$ be an absolutely almost simple, simply connected $F$-group which does not admit non-standard isogenies. This condition is always satisfied if ${\rm char}(F)>3$. More precisely, $\G$ admits non-standard isogenies only if ${\rm char}(F)=3$ and $\G$ is of type $G_2$ or ${\rm char}(F)=2$ and $\G$ is of type $B_n, C_n, (n\geq 2)$ or of type $F_4$, see~\cite[\S1]{Pink-Compact} and references therein.

Put $\bG'=\bG\times \bG$. Let 
\[
\bH=\{(h,h): h\in\G\}\subset \bG'
\]
be the diagonal embedding of $\bG$ in $\bG'$. We also write $\G^1=\G\times \{1\}$ and $\G^2=\{1\}\times \G$. For every group $\bf L$, we put $L={\bf L}(\adele)$. We will often simplify the notation and denote $G^1$ simply by $G$, but $G^2$ will always be denoted as $G^2$. 

Abusing the notation slightly, for every $g\in G'$ and a subset $B\subset G$, we let
\[
\Delta_g(B)=g\{(\disg,\disg): (\disg,1)\in B\}g^{-1}.
\] 
We adapt a similar notation for subsets of $\G(F_v)$ for all $v\in\Sigma$.

For $i=1,2$, let $\pi^i$ denote the projection from $\G'$ onto the $i$-th component. Denote by $\iota^1:\G\to\G'$ the map $g\mapsto (g,1)$ and $\iota^2$ is the map $g\mapsto (1,g)$. For every $v\in\Sigma$, we let $\pi_v: G'\to \G'(F_v)$ denote the natural projection, and $\iota_v:\G'(F_v)\to G'$ denotes the natural embedding in the $v$-th place. Put $\pi_v^i:=\pi^i\circ \pi_v$ and $\iota_v^i=\iota_v\circ\iota^i$.

For $L=G^i, G', H$ and every $v\in \Sigma$, let $L_v=\iota_v({\bf L}(F_v))$. Given $g=(g_v)\in G'$ and $w\in\Sigma$, we will let 
\[
g^w=\iota_w\circ\pi_w(g).
\]
That is, $g^w_w=g_w$ and $g^w_v=1$ for all $v\neq w$.

Let $\rho:\G\rightarrow \SL_N$ be an embedding defined over $F$.   
For any $v\in\Sigma$, let 
\[
K_v=\iota^1_v\bigl(\rho^{-1}(\SL_{N}(\frako_v))\bigr)\subset G_v
\]
and let $K\subset G$ denote the product of $K_v$ for all $v$. 
For $m\ge1$, let
\begin{equation} \label{Kvm definition}
K_v[m] := \mathrm{ker}(K_v \rightarrow \SL_N(\frako_v/\varpi_v^m \frako_v));
\end{equation}
note that $K_v= K_v$. Let 
\[
K_v' := \{(g^1,g^2): (g^i,1)\in K_v\}.
\]
Define $K'$ and $K_v'[m]$ accordingly. 
We choose $\Omega_0$ in~\eqref{eq: def volume} to be $K'$. 
For any $L\subset K_v$, $v\in\Sigma$, and $m\geq0$ we defined $L[m]:=L\cap K[m]$.

For any subspace $\frakh\subset\frakg_w$ we write $\frakh[0]$ for the preimage of the $\frako_w$-integral $N \times N$ matrices under the restriction of the differential $D\rho: \mathfrak{g} \rightarrow \mathfrak{sl}_N$ to $\frakh$.
More generally, we write $\mathfrak{h}[m]$ for the preimage of the matrices
all of whose entries have valuation at least $m$.

\medskip

Recall that $X=G'/\bG'(F)$ and $Y_g=gH/\H(F)$ where $g\in G'$. The homogeneous set $Y_g$ is equipped with the $H_g=gHg^{-1}$-invariant probability measure $\mu_g$.

 \begin{lemma}\label{lem:stabilizer}
    We have ${\rm Stab}(\mu_g)=H_g\cdot {\bf C}(F)$, where ${\bf C}$ denotes the center of $\G'$.
    \end{lemma} 
    
    \begin{proof}
        By virtue of~\cite[Lemma 2.2]{EMMV}, we have 
        \[
        {\rm Stab}(\mu_g)=g H{\bf N}(F)g^{-1} 
        \]
        where ${\bf N}$ denotes the normalizer of $\H$ in $\G'$. Since $\H=\{(h,h): h\in \G\}$ and $\G$ is absolutely, almost simple, we conclude that ${\bf N}=\H\cdot {\bf C}$ which implies the claim.  
    \end{proof}

\begin{lemma}
For every $g\in G'$, there exists some $\hat g\in G^2$ so that 
$Y_{g}=Y_{\hat g}$.
\end{lemma}

\begin{proof}
    Let $g=(g^{(1)}, g^{(2)})$, then 
    \begin{align*}
       Y_{g}&=\bigl\{\bigl(g^{(1)}, g^{(2)}\bigr)(h,h): h\in \G(\adele)\bigr\}\\
       &=\bigl\{\bigl(g^{(1)}h, g^{(2)}(g^{(1)}\bigr)^{-1}g^{(1)}h\bigr): h\in \G(\adele)\bigr\}\\
       &=\bigl\{(h, g^{(2)}(g^{(1)})^{-1}h): h\in \G(\adele)\bigr\}.
    \end{align*}
    The claim thus holds with $\hat g=\bigl(1, (g^{(2)}(g^{(1)})^{-1})_v\bigr)$.
    \end{proof}
    
    In view of this lemma, from this point until the end of the paper, we will always assume $g\in G^2$.

\subsection{Properties of a split Lie algebra}\label{sec: split Lie-alg}

Let $\mathfrak{g}$ be the Lie algebra of $\bG$ and for any $v\in\Sigma$, let 
$\mathfrak{g}_{v} := \Lie(\G) \otimes F_{v}$. 
We fix $w$ such that $\G$ is $F_w$-split. 
In particular, $G_w$ is a Chevalley group (see \cite[\S 3.4.2]{tits}). 
Let $\Phi$ be the set of roots for $G_w$. Let $p_w:={\rm char}(k_w)$.

Let $\fraka$ be a maximal split torus of $\frakg_w$. 
Then, $\frakg$ can be decomposed as $\frakg=\frakg_0\oplus\bigoplus_{\lambda\in\Phi}\frakg_\lambda$, where for any $\lambda\in\fraka^*$
\begin{equation}\label{eq:root space}
\frakg_\lambda:=\{v\in\frakg:\forall a\in\fraka,\ad(a)v=\lambda(a)v\},
\end{equation}
and $\Phi\subset\fraka^*$ is the set of roots for $\frakg_w$, i.e., the set of non-zero characters for which \eqref{eq:root space} is non-empty (see \cite[\S 2.6]{Steinberg}). 

For any $\lambda\in\Phi$, there exists a one-dimensional Lie algebra $\frakg_\lambda\subset\frakg$, and a one-dimensional unipotent subgroup $U_\lambda=\{u_\lambda(r)\}\subset G_w$ such that $\frakg_\lambda$ is the Lie algebra of $U_\lambda$ (see \cite[\S 2.3]{BoSp-Rationality}). 
Set $L_\lambda=\langle U_\lambda, U_{-\lambda}\rangle$. 
Then, for any $t$ we can also denote by $a_\lambda(t)$ the unique diagonal element in $L_\lambda$ 
such that \[a_\lambda(t) u_\lambda(1) a_\lambda(t^{-1})=u(t^2).\]

Fix two opposite Borel subgroups $B$ and $B^-$ in $G_w$. 
Denote by $U$ and $U^-$ the unipotent radicals of $B$ and $B^-$, respectively, and by $A=B\cap B^-$ 
a maximal split torus of $G_w$. For every nonnegative integer $m$, let 
\[
U_m=U\cap K_w[m],\quad U^-_m=U^-\cap K_w[m],\quad\text{and}\quad A_m=A\cap K_w[m].
\]
Then, according to \cite[\S 3.1.1]{Ti} for any $m\in\bbn$, we have 
\begin{equation}\label{eq: Km decomp}
    K_w[m]= U_m^-A_mU_m. 
\end{equation}
Fix an ordering $\lambda_1,\ldots, \lambda_d$ on $\Phi^+$ once and for all; we use the ordering $-\lambda_1,\ldots,-\lambda_d$ on $\Phi^-$. 
In view of~\cite[\S 3.1.1]{Ti} again, the product map from $\prod_{i=1}^d U_{\pm\lambda_i}[m]$ to $U^\pm[m]$ is a bijection. 
For every $\lambda\in \Phi$ and every $g\in U^\pm[1]$, we define $g_\lambda$ to be the $\lambda$-coordinate of $g$ in this identification. Using \eqref{eq: Km decomp}, we can define $g_\lambda$ for any $\lambda\in\Phi$ and $g=u^-au^+\in K_w[1]$ by taking $g_\lambda=u_\lambda^\pm$. 

For every non-negative integer $m$, we also define       
\[
A_{-m}=\{a\in A: aU_{2m}a^{-1}\subset U_0\;\;\text{and}\;\; a^{-1}U_{2m}^-a\subset U_0^-\}.
\]

\begin{lemma}\label{lem:Chevalley irred rep}
There exists some $m_0\in\bbn$ (depending only on the dimension of $\G$) such that the following holds. For every $g\in K_w[2m_0+1]$, there exists $\lambda\in \Phi$, so that
\[
\bigl|\bigl\{g'\in K_wA_{-m_0}K_w: \|(g'gg'^{-1})_\lambda-I\|\geq q_w^{-m_0}\|g'gg'^{-1}-I\|\bigr\}\bigr|\geq q_w^{-3m_0\dim\G}.
\]
\end{lemma}

\begin{proof}
We first note that for all $m\in\bbn$, all $g'\in K_wA_{-m}K_w$ and all $g\in K_w[2m+1]$, we have 
\[
g'gg'^{-1}\in K_w[1].
\]
Therefore $(g'gg'^{-1})_\lambda$ is defined.

Let us begin with the following claim: one can assume 
\begin{equation}\label{eq: g mu is substantial}
  \text{$\|g_\lambda-I\|\geq q_w^{-c} \|g-I\|\quad$ for some $\lambda\in \Phi$}  
\end{equation} 
where $c\in\bbn$ depends only on $\dim\G$. 

To see~\eqref{eq: g mu is substantial}, note that since $g\in K_w[1]$, we can write $g=g^-g^0g^+$ where $g^0\in A_1$, $g\in U_1$ and $g^-\in U^{-}_1$. 
Suppose now that~\eqref{eq: g mu is substantial} fails with $c=1$, then $\|g^0-I\|> \|g^\pm-I\|$. Let $\lambda\in \Phi$ be a simple root so that $|\lambda(g^0)|$ is maximal. The map 
\[
r\mapsto \bigl(u_\lambda(r)gu_\lambda(-r)\bigr)_{\lambda}
\]
is a rational function $\frac{f_1(r)}{f_2(r)}$ so that coefficients of $f_i$ are bounded by $\ll \|g-I\|$ and $|f_2(r)-1|\leq 1/2$ for all $|r|_w\leq 1$. Therefore, there exists some $r\in F_w$ with $|r|_w=1$ so that if we put $\hat g=u_{\lambda}(r)gu_\lambda(-r)$, then 
\[
\|\hat g_\lambda-I\|\gg \|\hat g-I\|.
\]
Now suppose the claim is proved for $\hat{g}$. For $\beta>0$ and $\star\in G$ put  
\[
E_\beta(\star)=\{g'\in K_wA_{-m_0}K_w: \|(g'\hat \star g'^{-1})_\lambda-I\|\geq \beta\|g'\star g'^{-1}-I\|\bigr\}.
\]
Then,
\[
E_\beta(\hat g)u_\lambda(r)\subset E_\beta(g)\quad\text{and} \quad |E_\beta(\hat g)u_\lambda(r)|=|E_\beta(\hat g)|\geq \beta^{3\dim\bG}.
\]

We thus assume that~\eqref{eq: g mu is substantial} is satisfied, and will show the lemma holds with $\lambda$ and $m_0=c+1$. For simplicity in the notation, let us assume $\lambda\in\Phi^+$, the argument when $\lambda\in\Phi^-$ is similar.  
Put $\tilde g=a_{\lambda}(\varpi_w^{-c-1})ga_{\lambda}(\varpi_w^{c+1})$. Then~$\tilde g\in K_w[1]$. 
Using \eqref{eq: Km decomp} and the decomposition of $U_1$ into one-dimensional unipotent subgroups, we can write
\[
g=g^{-}g^0h_1u_\lambda(s')h_2,\quad h_1, h_2\in U_1.
\]
We now want to look at the decomposition of  $\tilde{g}$ into a diagonal element and  elements of the one dimensional unipotent groups. 
Conjugation by $a_{\lambda}(\varpi_w^{-c-1})$ expend $u_\lambda(s')$ by $q_w^{2c+2}$, 
and for each other root $\sigma\neq\lambda$, it expand $g_\sigma$ by at most $q_w^{c+1}$.   
we conclude from~\eqref{eq: g mu is substantial} that $s=q_w^{2c+2}s'$ satisfies $u_\lambda(s)=\tilde{g}_\lambda$. Moreover, for $k$ such that $|s|_w=q^{-k}$, we have
\[
q_w^{-c-1}\leq |s|_w=q_w^{-k}\leq q_w^{2c+2}q_w^{-2m-1}=q_w^{-1},
\] 
and $\tilde g^0\in K_w[k+1]$, $\tilde g_{\sigma}\in K_w[k+1]$ for all $\sigma\neq \lambda$.

Since $K_w[k+1]$ is a normal subgroup of $K_w$, we have 
$\tilde g=h\tilde u_\lambda(s)$ where $h\in K_w[k+1]$. Moreover, for every $g_2\in K_w[k+1]$, we have 
\[
g_2\tilde gg_2^{-1}=h'' u_{\lambda}(s)
\]
where $h''\in K_w[k+1]$.

Altogether, there is some 
\[
g_1\in K_wA_{-c-1}K_w=K_wA_{-m_0}K_w
\]
so that for every $g_2\in K_w[m_0+1]$ we have 
$g_2g_1gg_1^{-1}g_2^{-1}\in K_w[1]$ and 
\[
 \|(g_2g_1gg_1^{-1}g_2^{-1})_\lambda -I\|\geq q_w^{-m_0}\|g-I\|.
\]
The proof is complete.
\end{proof}

\section{Construction of a good place}\label{sec: good place}
Recall that $Y_g=g\H(\adele)/\H(F)$, where $g=(1,(g_v))$.
In this section we show the existence of a place $w$ where the group 
\[
H_{g,w}=\Delta_g(G_w)=gH_wg^{-1}
\]
has {\em controlled geometry}.  

Recall the definition of $K_v$ and $K'_v$ from \S\ref{notation}. 
For every $v\in\Sigma$, let 
\begin{equation}\label{eq: def of KvY}
 K_{v,Y_g}:=H_{g,v}\cap K'_v=gH_vg^{-1}\cap K'_v=\Delta_g(K_v)\cap K'_v. 
\end{equation}
In this section, we will simply write $Y$ for $Y_g$, and denote $K_{v,Y_g}$ by $K_{v,Y}$.

\medskip

The following is the main result of this section.  

\begin{proposition}[Existence of a good place]\label{prop: splitting-place}
There exists a place $w$ of $F$ such that 
\begin{enumerate}
	\item $\G$ is split over~$F_w$. 
	\item $\pi_w(K_w')$ and $\pi_w(K_{w,Y})$ are hyperspecial subgroup of~$\G'(F_w)$ and $\pi_w(H_g)$, respectively,
	\item $g^w\in K_w'$, see \S\ref{sec:notation} for the notation. 
\item $q_w \ll \log(\vol Y))^{2}$ where the implied constant depends on $\G$.
\end{enumerate}
\end{proposition}

Let us begin by recalling some standard facts from Bruhat-Tits theory.

\subsection{Bruhat-Tits theory.}
\label{ggo}
Let $\G,F$ as in our setting (see \S\ref{sec:notation}), and let $v\in\Sigma$. Then 
\begin{enumerate}
\item For any point $x$ in the Bruhat-Tits building of $\G(F_v)$, 
there exists a smooth affine group scheme $\Gfrak_v^{(x)}$ over 
$\order_v$, unique up to isomorphism, such that: 
its generic fiber is $\G(F_v)$, and the compact open subgroup $\Gfrak_v^{(x)}(\order_v)$ is the stabilizer of $x$ in $\G(F_v)$, see~\cite[3.4.1]{Ti}. 
\medskip 
\item If $\G$ is split over $F_v$ and $x$ is a {\it special} point, then
the group scheme $\mathfrak G_v^{(x)}$ is a Chevalley group scheme with
generic fiber $\G$, see~\cite[3.4.2]{Ti}.
\medskip 
\item $\red_v:\Gfrak_v^{(x)}(\order_v)\rightarrow\underline{\Gfrak_v}^{(x)}(k_v)$, the reduction mod $\varpi_v$ map,  
is surjective, { which follows from the smoothness above}, see~\cite[3.4.4]{Ti}. \medskip
\item $\underline{\mathfrak G_v}^{(x)}$ is connected and semisimple if and only if $x$ is a {\it hyperspecial} point.  Stabilizers of hyperspecial points in $\G(F_v)$ will be called hyperspecial subgroups, see~\cite[3.8.1]{Ti} and~\cite[2.5]{pr}.
\end{enumerate}

If $\G$ is quasi-split over $F_v$, and splits over $\widehat{F_v}$ (the maximal unramified extension of $F_v$), then hyperspecial vertices exist; and they are compact open subgroups with maximal volume. Moreover a theorem of Steinberg implies that $\G$ is quasi-split over $\widehat{F_v}$ for all $v$, see~\cite[1.10.4]{Ti}.

It is known that for almost all $v$, 
the groups $K_v$ are hyperspecial, see~\cite[3.9.1]{Ti} (and~\S\ref{notation} for the definition of~$K_v$). 
We also recall that: for almost all $v$ the group $\G$ 
is quasi-split over $F_v$, see ~\cite[Thm.~6.7]{PlRap}.

\subsection{Adelic volumes and Tamagawa number}\label{s;volume-form}
Fix an algebraic volume form $\omega$ on $\G$ defined over~$F$. 
The form $\omega$ determines a Haar measure on each vector space 
$\mathfrak{g}_{v} := \Lie(\G) \otimes F_{v}$
which also gives rise to a normalization of the Haar measure on $\G(F_{v})$.  
We denote both these measures by $|\omega_{v}|$, and
denote by $|\omega_{\adele}|$ the product measure on $\G(\adele_{F})$.
Then
\begin{equation}\label{eq: tamagawa}
|\omega_{\adele}|(\G(\adele_{F}))/\G(F))=D_{F}^{\frac{1}{2}\dim\G}\tau(\G),
\end{equation}
where $\tau(\G)$ is the {\em Tamagawa number} of $\G$, and $D_{F}$ is the discriminant of ${F}$. 
In the case at hand $\G$ is simply connected, it was recently proved that $\tau(\G)=1$, see~\cite[\S1.3]{GatsLurie}.

\subsection{The quasisplit form}\label{sec: volumes}
The volume formula~\eqref{eq: tamagawa} relates the Haar measure on $Y$ to the algebraic volume form $\omega$
(and the field~$F$). However, the volume of our homogeneous set $Y$ as a subset 
of $X$ depends heavily on the amount of distortion coming from $g$.

Following~\cite[Sect.~0.4]{pr}, we let $\mathcal G$ denote 
a simply connected algebraic group defined and quasi-split over $F$ which is an inner form of $\G$. 
Let $L$ be the field associated\footnote{In most
cases~$L$ is the splitting field of~$\mathcal G$ except in the case where~$\mathcal G$ is a triality
form of~${}^6\mathsf{D}_4$ where it is a degree~$3$ subfield\label{pagewithfootnote}
of the degree~$6$ Galois splitting field with Galois group~$S_3$. Note that there are three such
subfields which are all Galois-conjugate.} to $\mathcal G$ as in~\cite[Sect.~0.2]{pr},
it has degree~$[L:F]\leq 3$. We note that $\mathcal G$ should be thought of as the {\em least distorted 
version of} $\G$.

Let $\omega^0$ be a differential form on $\mathcal G$ corresponding to $\omega$. This can be described as follows: Let $\varphi:\G\rightarrow\mathcal G$ be an isomorphism 
defined over some Galois extension of $F$. 
We choose $\omega^0$ so that $\omega'=\varphi^*(\omega^0)$, it is defined over~$F$.
It is shown in~\cite[Sect.\ 2.0--2.1]{pr} that, up to a root of unity of order at most 3, this is independent of the choice of $\varphi$. 

As was done in~\cite{pr}, we now introduce local normalizing parameters $c_{v}$ which scale the volume form $\omega^0$ to a more canonical volume form on $\mathcal G(F_{v})$. 

For every place $v\in\Sigma$, choose an $\mathfrak{o}_{v}$-structure  on $\mathcal G$, i.e., a smooth affine group scheme over $\order_{v}$ with generic fiber $\mathcal G$ as in \cite[Sect.\ 1.2]{pr}.

Let $\{\mathcal P_{v} \subset \mathcal G(F_{v})\}$ denote
a coherent collection of parahoric subgroups with ``maximal volume'', see~\cite[Sect.\ 1.2]{pr} for an explicit description.  Let us recall that by a coherent collection we mean that $\prod_{v\in \Sigma}\mathcal P_v$ is a compact open subgroup of $\mathcal G(\adele_{F})$. 
Note that any $F$-embedding of $\mathcal G$ into ${\rm GL}_{N}$ 
gives rise to a coherent family of compact open subgroups of $\mathcal G(\adele_{F})$ which at almost all places satisfies the above requirements on $\mathcal P_v$, see \S\ref{ggo}.
At the other places we may choose $\mathcal P_v$ as above and then use (1) in \S\ref{ggo} to define the $\order_{v}$-structure on $\mathcal G(F_{v})$.
Let us also remark that maximality of the volume implies that the corresponding parahoric is either hyperspecial (if $\mathcal G$ splits over an unramified extension)
or special with maximum volume (otherwise).

This allows us, in particular, to speak of ``reduction modulo $\varpi_{v}$". 
For every place $v$ of $F$, we let $\overline{\mathcal M}_{v}$ denote the reductive quotient of $\red_{v}(\mathcal P_v)$; this is a reductive group over the residue field. 

Let $r_{v}\in F_{v}$ be so that $r_{v}\omega^0_{v}$ is 
a form of maximal degree, defined over $\order_{v}$, whose reduction mod $\varpi_{v}$ is non-zero, and 
let $c_{v}=|r_v|_{v}$. 

\subsection{Product formula}\label{sec: p-formula}
Let us use the abbreviation~$D_{L/F}=D_LD_{F}^{-[L:F]}$
for the norm of the relative discriminant of~$L/F$, see~\cite[Thm.~A]{pr}.
It is shown in~\cite[Thm. 1.6]{pr} that  
\begin{equation}\label{eq: prod-lambda}
\prod_{v\in\Sigma} c_{v} = 
D_{L/F'}^{\frac{1}{2}\mathfrak{s}(\mathcal G)}\cdot A,
\end{equation}
where $A>0$ depends only on~$\G$ over~$\bar{F}$ and~$[F:\Q]$, $\mathfrak{s}(\mathcal G)=0$ when $\mathcal G$ splits
over $F$ in which case $L=F$, and $\mathfrak{s}(\mathcal G)\geq5$ otherwise; these constants depend only on the root system of $\mathcal G$.

It should be noted that the parameters $c_{v}$ were defined using $\mathcal G$ and~$\omega^0$ but will be used to renormalize $\omega_{v}$ on $\G$, and hence on $\H$.

Abusing the notation, we will speak of hyperspecial subgroups of $G_w$, $G'_w$, and $H_{g,w}$ in the remaining parts of this section.  

Let $\Sigma'\subset \Sigma$ denote the set of places where 
$\G$ is isomorphic to $\mathcal G$ over $F_v$ and that $K_v$ is hyperspecial. Then $\Sigma\setminus \Sigma'$ is finite.

\begin{lemma}\label{lem: Kv volume contribution}
Let $v\in \Sigma'$. Then 
\begin{enumerate}
    \item $c_v|\omega_v|(K_{v,Y})\leq 1$.
    \item Assume $g^v\not\in K_v'$. Then 
\[
c_v|\omega_v|(K_{v,Y})\leq 1/2
\]
\end{enumerate}
\end{lemma}

\begin{proof}
In view of the definition of $K_{v,Y}$, see~\eqref{eq: def of KvY}, we have 
\begin{equation}\label{eq: KvY lemma}
K_{v, Y}=\Delta_g(K_v)\cap K'_v.
\end{equation}

Since $v\in\Sigma'$, $\G$ is isomorphic to $\mathcal G$ over $F_v$ and we have $c_v|\omega_v|(K_v)=1$ which implies the first claim.

To see the second claim, first note that if $g^v$ does not normalize $K_v'$, then $g_v$ does not normalize $\pi_v^2(K_v)$, where $\pi_v^2: G_v\to \G(F_v)$ is the natural projection, see \S\ref{sec:notation}. In particular,   
\[
[\pi_v^2(K_v): \pi_v^2(K_v)\cap g_v^{-1}\pi_v^2(K_v)g_v]\geq 2.
\] 
Now since $v\in\Sigma'$, we have $\pi_v^2(K_v)$ is a hyperspecial vertex in $\G(F_v)$, $\pi_v^2(K_v)$ equals its normalizer in $\G(F_v)$; this and the above imply part~(2) in view of~\eqref{eq: KvY lemma}.

We recall the argument that $\pi_v^2(K_v)$ equals its own normalizer for the convenience of the reader; $\pi_v^2(K_v)$ is the stabilizer of a hyperspecial point, $\mathsf{x}_v$ say. Since $g_v$ normalizes $\pi_v^2(K_v)$, $\pi_v^2(K_v)$ also stabilizes of $g_v\mathsf{x}_v$. If $g_v\mathsf{x}_v\neq \mathsf{x}_v$, then $\pi_v^2(K_v)$ stabilizes the geodesic between these two points which is a contradiction.   
\end{proof}

\subsection{The volume of a homogeneous set}\label{sec: vol-hom}
In view of our definition of volume 
and taking into account the choice of $\Omega_0$, 
the equation~\eqref{eq: tamagawa} implies that  
\begin{equation}\label{eq: vol-est}
\vol(Y)  =  D_{F}^{\frac{1}{2}\dim\G} \prod_{v\in\Sigma} 
\left(|\omega_v|(K_{v,Y})\right)^{-1},
\end{equation}
where $|\omega_v|$ is as before. This and~\eqref{eq: prod-lambda} imply that
\begin{equation}\label{eq;volume-upperbd}
\vol(Y) =A D_{L/F}^{\frac{1}{2}\mathfrak{s}(\mathcal G)}D_{F}^{\frac{1}{2}\dim\G} \prod_{v\in\Sigma}\Big(c_v|\omega_v|(K_{v,Y})\Big)^{-1}.
\end{equation}
Let us note the rather trivial consequence\footnote{This would also follow trivially
from the definition if only we would know that the orbit intersects a fixed compact subset.}~$\vol(Y)\gg 1$ of~\eqref{eq;volume-upperbd}. 
Below we will assume implicitly~$\vol(Y)\geq 2$ (which we may achieve by replacing~$\Omega_v$ by a smaller neighborhood 
at one place in a way that depends only on~$\G$).

\begin{proof}[Proof Proposition~\ref{prop: splitting-place}]
Recall that $\Sigma'\subset \Sigma$ denotes the set of places where 
$\G$ is isomorphic to $\mathcal G$ over $F_v$ and that $K_v$ and hence $K'_v=\{(\disg_1,\disg_2): (\disg_i,1)\in K_v\}$ is hyperspecial. 
Furthermore, if $g^v\in K_v'$, then 
\[
K_{v, Y}=\Delta_g(K_v)\cap K'_v
\]
is hyperspecial in $H_{g,v}$.

Now recall from Lemma~\ref{lem: Kv volume contribution} that for all places $v\in\Sigma'$, we have $c_v|\omega_v|(K_{v,Y})\leq 1$.
Moreover, if $v\in \Sigma'$ 
but $g^v\not\in K_v'$, then 
\begin{equation}\label{eq: index Kv cap K'v}
c_v|\omega_v|(K_{v,Y})\leq 1/2.
\end{equation}

By Chebotarev density theorem, the set 
\[
\Sigma''=\{v\in\Sigma': \G\text{ is split over $F_v$}\}
\]
has positive density, $c>0$ say, which depends only on $\G$. 

Let $B$ be a constant depending on $\G$ which will be explicated later. 
Assume now (for a contradiction) that for every $v\in \Sigma''$ with $q_v\leq Bc^{-1}\log(\vol(Y))^2$, 
we have $c_v|\omega_v|(K_{v,Y})^{-1}\geq 2$. 

Recall also that we are assuming $\vol(Y)\geq 2$. The above discussion,~\eqref{eq;volume-upperbd}, 
and the prime number theorem imply that 
\[
\vol(Y)\geq B' 2^{B\log(\vol(Y))^2/(\log B -\log c+2\log\log(\vol(Y)))}
\]
where $B'$ depends only on $\G$.
Now if $B\gg_\G 1$, then 
\[
B' 2^{B\log(\vol(Y))^2/(\log B -\log c+2\log\log(\vol(Y)))}>\vol(Y).
\]
This contradiction finishes the proof. 
\end{proof}

\section{Deduction of Theorem \ref{thm: property tau} from Theorem \ref{thm;equidistribution}}

In this section, we complete the proof of Theorem~\ref{thm: property tau}. The proof relies on Theorem~\ref{thm:prop-tau} below, which in turn is a consequence of Theorem~\ref{thm:main}. It should be noted, however, that the proof of Theorem~\ref{thm:main} requires a certain {\em uniform} decay of matrix coefficients for $H_{g,w}$ where $w$ is a good place, see~\S\ref{sec:sob-intro}. In the setting of Theorem~\ref{thm:prop-tau}, the group $H_{g,w}$ has Kazhdan's property $(T)$, thus the required uniform rate follows from the work of Oh~\cite{HeeOh-T}. 

\medskip

As before, we let $\pi^1$ denote the projection onto the first factor. 
Throughout this section, the space $\G(\adele)/\G(F)$
is equipped with the $\G(\adele)$-invariant probability measure, we thus drop this measure from the notation.

\begin{theorem}\label{thm:prop-tau}
    Assume the absolute rank of $\G$ is at least 2. There exists $\constk\label{k:thm-tau}>0$ which satisfies the following.
	Let $f_1, f_2\in L^2_0(\G(\adele)/\G(F))$ ---
	the orthogonal complement of $\G(\adele)$-invariant functions --- be $\pi^1(K_v)$-finite functions. Then  
	\[
	\left|\langle \disg.f_1,f_2\rangle\right|\le\dim\langle
	 \pi^1(K_v).f_1\rangle^{1/2}\dim\langle
	 \pi^1(K_v).f_2\rangle^{1/2}\Xi_{\G(F_v)}(\disg)^{-\ref{k:thm-tau}}\norm{f_1}_2\norm{f_2}_2,\]
	where $\disg\in\G(F_v)$, $\langle \pi^1(K_v).f\rangle$ is the linear span of $\pi^1(K_v).f$, and $\Xi_{\G(F_v)}$ is the Harish-Chandra spherical function of $\G(F_v)$. 
\end{theorem}

Let us first assume Theorem~\ref{thm:prop-tau} and finish the proof of Theorem~\ref{thm: property tau}. 

\begin{proof}[Proof of Theorem~\ref{thm: property tau}]
First note that the theorem in the case where $\G$ is of type $A_1$ follows from~\cite{Drinfeld}. We may, therefore, assume that the absolute rank of $\G$ is at least 2. 

By \cite{CHH1988}, see also~\cite[Lemma 3.22]{Gorodnik-Maucourant-Oh-adelic}, 
we may restrict ourselves to smooth functions in $L^2_0(\G(\adele)/\G(F))$. 
Now by \cite[Thm.~2 and Cor.]{CHH1988}, Theorem \ref{thm: property tau} is equivalent to the existence of a uniform decay rate for matrix coefficients for $\pi^1(K_v)$-finite functions in $L^2_0(\G(\adele)/\G(F))$ as claimed in Theorem~\ref{thm:prop-tau}. 
\end{proof}

\begin{proof}[Proof of Theorem~\ref{thm:prop-tau}]
Let $g=(1, (g_u))\in G'$ where $g_{v}=\disg$ and $g_{u}=1$ for all $u\neq v$.

In view of our definition of the volume of a homogeneous set,
there exist positive constants~$\constk\label{k:kapppa1}$ and $\constk\label{k:kapppa2}$ (depending only on the root system of~$\bG(F_v)$) such that  
\[
\|g\|^{\ref{k:kapppa1}}\ll \vol(Y_g)\ll \|g\|^{\ref{k:kapppa2}}.
\]

Let $w$ be a place obtained by applying Proposition~\ref{prop: splitting-place} to $Y_g$. 
    In particular, 
    \begin{enumerate}
	\item $\G$ is split over~$F_w$. 
	\item $K_w'$ and $K_{w,Y}$ are hyperspecial subgroup of~$G_w$ and $H_w$, respectively 
\item $q_w \ll \log(\vol Y_g))^{2}$ where the implied constant depends on $\G$.
    \end{enumerate}
    
Then, by \cite[Thm.~1.1]{HeeOh-T}, and our assumption that the absolute rank of $\bG$ is at least two, the group $G_w$ has property $(T)$. 
Therefore, there exists a constant $\constk\label{k:propT}$, (which depends only on the type of $\G$) so that for all  $K_w$-finite functions
$f_1,f_2\in L^2_0(\G(\adele)/\G(F))$ and all $h_w\in G_w$, we have
\begin{equation}\label{eq:oh-mc}
|\langle h_w.f_1,f_2\rangle|\ll \dim\langle \pi_1(K_w).f_1\rangle^{1/2}\dim \langle \pi_1(K_w).f_2\rangle^{1/2}\|f_1\|_2\|f_2\|_2\|h_w\|^{-\ref{k:propT}}
\end{equation}
where the implicit constant depends on $G_w$. If ${\rm char}(F)>2$ this is stated in~\cite[Thm.~1.1]{HeeOh-T}. A careful examination of the proof shows that \eqref{eq:oh-mc} holds when ${\rm char}(F)=2$ as well, see~\cite[\S 4.7]{HeeOh-T}.  

Using~\eqref{eq:oh-mc} as an input in the proof of
Theorem~\ref{thm:main}, see~\S\ref{sec:sob-intro}, we conclude from Theorem~\ref{thm:main} that if $f=f_1\otimes\bar f_2$
with $f_i\in C_c^\infty(\G(\adele)/\G(F))$ for $i=1,2,$ then 
\begin{align*}
\Big|\int_{X} f\operatorname{d}\!\mu_{Y_g}&-\int_{X} f_1\otimes\bar f_2\operatorname{d}\!m_X\Big|\\
&=\Big|\langle g.f_1, f_2\rangle_{\G(\adele)/\G(F)}-\int_{\G(\adele)/\G(F)} f_1\int_{\G(\adele)/\G(F)}\bar f_2\Big|\ll \|g\|^{-\kappa}\Sob(f).
\end{align*}

The implied multiplicative constant depends on~$\G(\adele)/\G(F)$ and so also on the~$F$-structure of~$\G$.
We note however, that thanks to~\cite{CHH1988}, 
the above upgrades to a uniform effective bound on the decay of the matrix coefficients as in Theorem~\ref{thm:prop-tau}
with $\ref{k:thm-tau}$ independent of the $F$-form of~$\G$. This implies Theorem~\ref{thm:prop-tau}.
\end{proof}

\section{Effective generation of a group at a good place}\label{sec:effective generation}
In this section, we fix a good place $w$ and show some useful properties of $G_w$, 
see \S \ref{sec: good place} for the definition and the proof of the existence of a good place.   
In particular, $\G$ is split over $F_w$ and we can apply the notation and results from \S\ref{sec: split Lie-alg}. 

Recall that $\frako_w$ is the maximal compact subring of $F_w$, $\varpi_w$ is a uniformizer of $\frako_w$, 
and $\pi^1(K_w)=\mathbf{G}(\frako_w)$. 
For any $L>0$ and $R\subset G$, denote by $R[L]$ the set of elements in $R$ whose entries vanish in $\frako_w/\varpi_w^L\frako_w$. 
For any $L>0$ and $\frakr\subset\frakg$ denote by $\frakr[L]$ the set of elements in $\frakr$ whose entries have valuation at least $L$.

For every $\lambda\in\Phi$, let $A_\lambda$ be the scheme theoretic split torus in $\langle U_\lambda,U_{-\lambda}\rangle	\cong\operatorname{SL}_2$ which we realize as a closed $\bbz$-subscheme of $\G$. Put
	\begin{equation}\label{eq: definition of A'}
	A'_\lambda=\Delta_g(A_\lambda) \subset H_g
	\end{equation}
	the elements in $A'_\lambda$ will be denoted by $a_\lambda'(t)$.
	Similarly, put $U'_{\pm\lambda}=\Delta_g(U_{\pm\lambda})\subset H_g$ and their elements by $u'_{\pm\lambda}(s)$. 
	
	Note that for every $\disg\in G_w$, $t\in F_w^\times$, and $s\in F_w$, we have 
	\begin{equation}\label{eq: a and a' conjugation}
	   a'_\lambda(t)\disg a'_\lambda(t)^{-1} = a_\lambda(t)\disg a_\lambda(t)^{-1}\;\text{ and }\;\; u'_\lambda(s)\disg u'_\lambda(-s) = u_\lambda(s)\disg u_\lambda(-s).
	\end{equation}

For $b\in F_w$ and $r>0$, let 
\[
B(b,r):=\left\{s\in F_w:|b-s|_w<r\right\}. 
\]

\begin{lem}\label{lem:implicit function}
There exists an absolute constant $\constk\label{k:implicit func}$ such that the following holds. 
\begin{enumerate}
    \item Assume that ${\rm char}(F)>2$. For any $\lambda\in\Phi$ and every $|s|_w=1$, we have 
    \[
U_\lambda[\ref{k:implicit func}]\subset \{u_\lambda(r^2-s^2):|r|_w=1\}.
\]
\item Assume that ${\rm char}(F)=2$. Then 
\[
\{u_\lambda(c^2): |c|_w\leq1\}\subset \{u_\lambda(r^2-s^2):|r|_w=1\}.
\]
\end{enumerate}
\end{lem}

\begin{proof} 
    Let us first assume that ${\rm char}(F)>2$. Let $\psi_s:F_w^2\rightarrow F_w$ be defined by $\psi_s(t,r)=t-(r^2-s^2)$. Then, 
	\[
	\frac{\partial \psi_s}{\partial r}(0,1)=-2\neq 0.
	\] 
	Hence, by the implicit function theorem (see \cite[\S 4]{IFT}), there exists $0<\alpha< 1$, such that for all $t\in B(0,\alpha^2)$ there exists $r\in B(s,\alpha)$ so that $\psi_s(t,r)=0$, i.e., $t=r^2-s^2$. 
	This implies the first claim with $\ref{k:implicit func}=\alpha^2$.  
	
	If ${\rm char}(F)=2$, then $r^2-s^2=(r-s)^2$ which implies the claim.  
\end{proof}

\begin{lem}\label{lem:effectively generation}
Recall that $\G$ does not admit any non-standard isogenies.
There exist $\constk\label{k: Gw generation}$ 
and $\ell\geq 1$ with the following property. Let $\lambda\in\Phi$. 
Then, there exist $h_1,\ldots, h_\ell\in H_{g,w}$ with $\|h_i\|\leq 1$ so that 
every $\disg\in K_w[\ref{k: Gw generation}]$ can be written as
\[
\disg = \disg_1\disg_2\cdots \disg_{\ell}, \]
where for any $1\le i\le \ell$, $\disg_i\in h_i U_\lambda[0]h_i^{-1}$. 
\end{lem}

\begin{proof}
Clearly it suffices to prove this for a fixed $\lambda\in\Phi$. 
First note that since $G=G^1$ (see \S\ref{sec:notation}) for every $h=(h^1,h^2)\in H_{g,w}$, we have 
\[
hU_\lambda h^{-1}=(h^{1},1)U_\lambda (h^{1},1)^{-1}\subset G_w\quad\text{and}\quad \Ad(h)\mathfrak u_\lambda=\Ad((h^{1},1))\mathfrak u_\lambda\subset \gfrak_w.
\]

Since $\G$ is simply connected and does not admit any non-standard isogenies, there exist 
$h'_1,\ldots,h'_{\dim\G}\in H_{g,w}$ so that the $F_w$-span of   $\{\Ad(h'_i)\mathfrak u_\lambda\}$ equals $\gfrak_w$, see \cite[Prop.~1.11]{Pink-Compact} and references therein. 
Moreover, since $K_w$ is Zariski dense in $G_w$, we can choose $h'_i$ so that $\|h'_i\|\leq 1$ for all $i$.

In consequence, by \cite[Prop.~1.16]{BoSp-Rationality} applied  with 
$\{h_i'U_\lambda {h_{i}'}^{-1}: 1\leq i\leq \dim\mathbf G\}$
and $G_w$, we conclude that the product map 
\[
\varphi:\prod_{j=1}^\ell U_j\to G_w
\]  
is a separable surjective morphism, where $\ell$ is absolute and $U_j\in\{h'_iU_\lambda {h'_{i}}^{-1}: 1\leq i\leq \dim\mathbf G\}$ for every $1\leq j\leq \ell$.   

Thus, by the implicit function theorem (see \cite[\S3]{IFT}), there exists $\ref{k: Gw generation}$ so that 
\[
G_w[\ref{k: Gw generation}]\subset\varphi\Bigl(\textstyle\prod_{j=1}^\ell U_j[0]\Bigr)
\]
where $U_j[0]:=h'_jU_\lambda[0] {h'_{j}}^{-1}$. The proof is complete. 
\end{proof}

\begin{remark}
{\em Lemma~\ref{lem:effectively generation} is the only place in the paper where the assumption that $\G$ does not admit non-standard isogenies is used.} 
\end{remark}

We also need the following lemma, which deals with the case that the characteristic of $F$ equals 2, see Lemma~\ref{lem:implicit function}.  

\begin{lemma}\label{lem:char 2}
Assume ${\rm char}(F)=2$ and let $\tau_0\in\order_w$ be nonzero. Let $\lambda\in\Phi$. There exists an absolute constant $\ell'$ with the following property. For every $r\in B(0,1)$, 
\[
u_\lambda(r)= h_1\cdots h_{\ell'}
\]
where for all $1\leq i\leq \ell'$, either $h_i\in \{u_{\lambda}(\tau_0\alpha^2): |\tau_0\alpha^2|\leq |\tau_0|^{-1}\}$ or $h_i\in H_{g,w}$ with $\|h_i\|\leq 1$. 
\end{lemma}

\begin{proof}
We give a hands on proof based on direct computations.  
Let us begin with the following observation: 
Let $\varphi_\lambda: \SL_2 \to \langle U_\lambda, U_{-\lambda}\rangle$ be the natural isomorphism which maps the upper (resp.\ lower) triangular unipotent subgroup of $\SL_2$ to $U_\lambda$ (resp.\ $U_{-\lambda}$). There exist $0<\constk\label{k:char 2}<1$ and $\mathsf h_1, \mathsf h_2, \mathsf h_3\in H_{g,w}$ with $\|\mathsf h_i\|\leq 1$ so that every element
 \[
    h\in \varphi_\lambda\left(\begin{pmatrix}\tau_0 & 0\\ 0& 1\end{pmatrix}\SL_2(\order_w^2[\ref{k:char 2}])\begin{pmatrix}\tau_0^{-1} & 0\\ 0& 1\end{pmatrix}\right)=:L'_{\lambda, \tau_0}[\ref{k:char 2}]
    \]
    can be written as a product 
    \begin{equation}\label{eq: char 2 sl2}
    h=h_1\cdots h_6
    \end{equation}
    where for every $1\leq j\leq 6$, $h_j\in \mathsf h_i \{u_{\lambda}(\tau_0\alpha^2): |\tau_0\alpha^2|\leq 1\}\mathsf h_i^{-1}$ for some $1\leq i\leq 3$.
    
    This is a special case of Lemma~\ref{lem:effectively generation}, applied with $\SL_2$ over $F_w^2$ and the root subgroup $\{u_\lambda(r): r\in F_w^2\}$.

 We now turn to the proof of the lemma. Let $|\mathsf a|, |\mathsf c|\leq \ref{k:char 2}$ and put $r'=\frac{r}{1+\mathsf a^4}$, then 
 \[
 \begin{pmatrix}r' & 1\\1& 0\end{pmatrix}
 \begin{pmatrix}\mathsf a^2 & 0\\
\tau_0^{-1} \mathsf c^2& \mathsf a^{-2}\end{pmatrix}
\begin{pmatrix}0 & 1\\
1& r'\end{pmatrix}=\begin{pmatrix}a^{-2} & r'\mathsf a^2+r'\mathsf a^{-2}+\tau_0^{-1}\mathsf c^2\\
0& \mathsf a^2\end{pmatrix}.
 \]
 This implies that 
 \begin{multline*}
 \Delta_g\left(\varphi_\lambda\left(\begin{pmatrix}r' & 1\\1& 0\end{pmatrix}\right)\right)
 \varphi_\lambda\left(\begin{pmatrix}\mathsf a^2 & 0\\
\tau_0^{-1} \mathsf c^2& \mathsf a^{-2}\end{pmatrix}\right)
\Delta_g\left(\varphi_\lambda\left(\begin{pmatrix}0 & 1\\
1& r'\end{pmatrix}\right)\right)=
\\\varphi_\lambda\left(\begin{pmatrix}a^{-2} & r'\mathsf a^2+r'\mathsf a^{-2}+\tau_0^{-1}\mathsf c^2\\
0& \mathsf a^2\end{pmatrix}\right).
 \end{multline*}
Now, we have 
 \[
 \begin{pmatrix}\mathsf a^2 & 0\\0& \mathsf a^{-2}\end{pmatrix}
 \begin{pmatrix}\mathsf a^{-2} & r'\mathsf a^2+r'\mathsf a^{-2}+\tau_0^{-1}\mathsf c^2 \\0& \mathsf a^2\end{pmatrix}
 \begin{pmatrix}1 & \tau_0 (\tau_0^{-1}\mathsf a\mathsf c)^2\\0& 1 \end{pmatrix}=
 \begin{pmatrix}1 & r'(1+\mathsf a^4)\\0& 1 \end{pmatrix};
 \]
 applying $\varphi_\lambda$, we conclude that 
 \[
 \varphi_\lambda\left(\begin{pmatrix}\mathsf a^2 & 0\\0& \mathsf a^{-2}\end{pmatrix}\right)\varphi_\lambda\left(\begin{pmatrix}a^{-2} & r'\mathsf a^2+r'\mathsf a^{-2}+\tau_0^{-1}\mathsf c^2\\
0& \mathsf a^2\end{pmatrix}\right)u_\lambda\bigl(\tau_0 (\tau_0^{-1}\mathsf a\mathsf c)^2\bigr)=u_\lambda(r),    
 \]
where we also used $r=(1+\mathsf a^4)r'$. 
 
 The claim follows from these computations and~\eqref{eq: char 2 sl2}.
\end{proof}

\begin{lemma}\label{lem:sl2-rep}
There exist $\constE\label{C:sl2-rep}$ so that the following holds. 
Let $m_1\in\bbn$, $m\geq 3m_1$, and $\lambda\in\Phi$. 
Let $g\in K_w[m]$ satisfy $\|g_\lambda-I\|\geq q_w^{-m_1}\|g-I\|$.
Let $m_g\geq m$ be so that $g\in K_w[m_g]\setminus K_w[m_g+1]$.
For every $q_w^{m_1}\leq |t|_w\leq q_w^{(m_g-m_1)/2}$, and all $|s|_w\leq1$, we have 
\[
\Bigl\|\al'(t)\ul'(s)g\ul'(-s)\al'(t^{-1})-\al(t)g_\lambda\al(t^{-1})\Bigr\|\leq \ref{C:sl2-rep} q_w^{m_1} \max\{|s|_w,|t|_w^{-1}\}. 
\]
If we further assume that $|t|_w=q_w^{(m_g-m_1)/2}\geq \ref{C:sl2-rep}q_w^{2m_1+3}$ and $|s|_w\leq \ref{C:sl2-rep}^{-1}q_w^{-2m_1-3}$, 
\[
\|\al'(t)\ul'(s)g\ul'(-s)\al'(t^{-1})-I\|\geq q_w^{-m_1-2},
\]
where the implied constant is absolute. 
\end{lemma}

\begin{proof}
First note that in view of~\eqref{eq: a and a' conjugation}, we may replace $a'_\lambda$ and $u'_\lambda$ with $a_\lambda$ and $u_\lambda$, respectively. 

The proof is based on the following observation: let $g\in K_w[m']$ for some $m'\geq 1$, then for all $s\in B_F(0,1)$ we have 
\[
\ul(s)g\ul(-s)=gh(s,g)
\]
where $\|h(s,g)\|\ll \|g-I\||s|_w$ and the implied constant is absolute.  

Let $g$ and $t$ be as in the statement, and let $s\in B_w(0,1)$. 
Conjugating the above with $\al(t)$, we conclude that
\begin{equation}\label{eq:h-t-s-g}
   \al(t)\ul(s)g\ul(-s)\al(t^{-1})=\al(t)g\al(t^{-1})h_t(s,g) 
\end{equation}
where where $h_t(s,g)=\al(t)h(s,g)\al(t^{-1})$, hence, 
\[
\|h_t(s,g)-I\|\ll q_w^{m_g}\|g-I\||s|_w\ll |s|_w
\]
for an absolute implied constant. 

Recall the condition $\|g_\lambda-I\|\geq q_w^{-m_1}\|g-I\|$. 
Moreover, note that the conjugation by $\al(t)$ has weight $2$ for $U_\lambda$ and weight $\leq 1$ otherwise. We thus conclude that 
\[
\|\al(t)g\al(t^{-1})-\al(t)g_\lambda\al(t^{-1})\|\leq q_w^{m_1}|t|_w^{-1}
\]
This and~\eqref{eq:h-t-s-g} finish the proof of the first claim. 

To see the second claim first note that $a_\lambda(t)g_\lambda a_{\lambda}(t^{-1})=u_\lambda(\tau)$ where $|\tau|\geq q_w^{-m_1-2}$; moreover, if $|t|_w=q_w^{(m_g-m_1)/2}\geq \ref{C:sl2-rep}q_w^{2m_1+3}$ and $|s|_w\leq \ref{C:sl2-rep}^{-1}q_w^{-2m_1-3}$, then 
\[
\ref{C:sl2-rep} q_w^{m_1} \max\{|s|_w,|t|_w^{-1}\}\leq q_w^{-m_1-3}.
\]
The second claim thus follows from the first claim and the above. 
\end{proof}

Recall from \S\ref{sec: good place} that we fixed a Haar measure $\prod c_v|\omega_v|$ on $H_g$, and in particular, $c_w|\omega_w|$ on $H_{g,w}$. Unless otherwise is stated explicitly, $c_w|\omega_w|$ 
is the measures of reference on $H_{g,w}$ throughout the paper.

\begin{lemma}\label{lem:adjustment-1}
Let $m_0$ be as in Lemma~\ref{lem:Chevalley irred rep}. 
Let $m\geq 6m_0+1$. 
For $i=1,2$, let 
\[
E_i\subset K_wA_{-m_0}K_w
\]
be two subsets so that $|K_wA_{-m_0}K_w\setminus E_i|\leq  q_w^{-4m_0\dim\G}$.

Then for every $\disg'\in K_w'[m]$, there exist $\alpha_i\in \Delta_g(E_i)$ and $\lambda\in\Phi$ so that 
\[
\alpha_2 \disg'\alpha_1^{-1}=\disg \in K_w[m-2m_0]\subset G_w\times \{1\}
\]
$\|\disg_{\lambda}-I\|\gg \|\disg-I\|$ where the implied constant depends only on $\dim\G$.
\end{lemma}

\begin{proof}
Let us write $\disg'=h \hat\disg$ for some $h\in H_{g,w}\cap K'_w[m]$ and $\hat\disg\in K_w[m]=G_w\cap K'_w[m]$. 
If $\alpha \in \Delta_g(K_wA_{-m_0}K_w)$, then
\[
\disg'\alpha^{-1} = (h\alpha^{-1}) \alpha \hat\disg\alpha^{-1}.
\]
Note that since $\hat\disg\in K_w[m]$, for all $\alpha$ as above we have $\alpha \hat\disg\alpha^{-1}\in K_w[m-2m_0]$. We will find $\alpha_1,\alpha_2\in \Delta_g(E_i)$ so that 
$\alpha_2 \disg'\alpha_1^{-1}= \alpha_1 \hat\disg\alpha_1^{-1}$.

Since $\hat\disg\in K_w[m]\subset K_w[2m_0+1]$, in view of Lemma~\ref{lem:Chevalley irred rep}, there exists $\lambda\in\Phi$ so that
\[
|E_{\hat\disg}|\geq q_w^{-3m_0\dim\G}
\]
where 
$E_{\hat\disg}=\bigl\{y\in K_wA_{-m_0}K_w: \|(y \hat\disg y^{-1})_\lambda-I\|\geq q_w^{-m_0}\|y \hat\disg y^{-1}-I\|\bigr\}$.

Since $h\in K'_w\cap H_{g,w}$, we have 
\[
h.\Delta_g(K_wA_{-m_0}K_w)=\Delta_g(K_wA_{-m_0}K_w).
\]
In particular, the map $f:\Delta_g(E_1)\to \Delta_g(K_wA_{-m_0}K_w)$
defined by $\alpha \mapsto h\alpha^{-1}$ is measure-preserving. 

Altogether, and in view of our assumption on the relative measures of $E_1$ and $E_2$, we may choose 
$\alpha_1 \in \Delta_g(E_1)\cap \Delta_g(E_{\hat\disg})$ 
with $\alpha_2=f(\alpha_1)^{-1}  \in \Delta_g(E_2)$. 
The lemma follows.
\end{proof}

\section{Non-divergence of unipotent orbits}\label{s;nondivergence}

In this section we show that when $X$ is not compact, the cusps have small measure with respect to $\mu_g$. We follow the proof of \cite[Lemma 3.6.1]{EMV}, which relies on the non-divergence of unipotent flows. 

Let $\mathbb F$ be a finite field such that $F$ is a finite separable extension of $\tF:=\mathbb{F}(T)$, the field of rational functions in one variable over $\mathbb{F}$, see~e.g.~\cite[Ch.~III]{Weil-BasicNum}.  
Let $\tilde{\G}={\rm Res}_{F/\tF}(\G
')$. Then $\tilde{\G}$ is a semisimple $\tF$-group and 
\[
\tilde{\G}(\adele_{\tF})/\tilde{\G}(\tF)=\G'(\adele_{F})/\G'(F).\]

Fix a good place $w\in \Sigma$, and let $\tilde w$ be a place of $\tF$ so that $w|\tilde w$. We note that since $w$ is a good place for $\G'$, the group $\tilde\G$ is isotropic over $\tF_{\tilde w}$ (however, it need not be split).

Let $K'(\tilde w):=\prod_{v\nmid \tilde w} K'_v$. Since $\bG$ is simply connected, it follows from strong approximation (see \cite{strong_approx}) that one can write 
\[
\G'(\adele)= K'(\tilde w)\Bigl(\textstyle\prod_{v\mid \tilde w}G_v\Bigr)\G'(F)=K'(\tilde w) \tilde \bG(\tF_{\tilde w})\tilde \bG(\tF).
\]
Let $\Gamma=\tilde\bG(\tF)\cap K'(\tilde w)$ and $\tGw=\tilde\G(\tF_{\tilde w})$; note that $\Gamma$ is a congruence subgroup of $\tGw$. The above implies that
\[
K'(\tilde w)\backslash\G'(\adele)/\G'(F)\cong \tGw/ \Gamma.
\]

Let $\mathcal O$ be the ring of $\tilde w$-integers in $\tF$; note that $\mathcal O$ is a PID. Let $\tilde{\frakg}=\Lie(\tGw)$, and let $\tilde{\frakg}_{\mathcal O}$ be an $\mathcal O$-lattice in the $\tilde{\mathfrak{g}}$ with the property that  $[\tilde{\mathfrak{g}}_{\mathcal O}, \tilde{\mathfrak{g}}_{\mathcal O}] \subset \tilde{\mathfrak{g}}_{\mathcal O}$. As it is done in \cite{EMV}, for any $x\in X$, we set
\begin{align*}
    \height(x)=\sup\bigl\{\norm{\Ad(g)\mathbf u}^{-1}:x=g\Gamma,\:g\in \tGw,\:\mathbf u\in\tilde{\frakg}_{\mathcal O}\setminus\{0\}\bigr\},
\end{align*}
and 
\[
\Siegel(R):=\{x\in X :\height (x)\leq R\}.
\]

\begin{lemma}\label{lem: inj radius}
For every $m\in\bbn$, let 
\[
Q'_m:=\bigl\{\bigl((g_v^1),(g^2_v)\bigr)\in K': (g_w^1, g_w^2)\in \pi_w(K'_w[m])\bigr\}.
\]
There exists $\constk\label{k: inj rad}$ so that for all 
$x\in X$ with $m\leq \height(x)^{-\ref{k: inj rad}}$, the map
$\disg\mapsto \disg x$ is injective on $Q'_m$.
\end{lemma}

\begin{proof}
Fix some $m\in\bbn$ and assume that $\disg_1 x= \disg_2 x$ for $\disg_1, \disg_2\in Q'_m$. Fix $\tilde g \in\tGw=\tilde\bG(\tilde F_{\tilde w})$ such that the projection to the $\tilde w$ place of $x$ is $\tilde g\Gamma$.  Then, $\disg_{1,\tw}\tilde g\Gamma =\disg_{2,\tw}\tilde g\Gamma$, and so  $\disg_{2,\tw}^{-1}\disg_{1,\tw}$ fixes $L_x:=\Ad(\disg) \tilde{\mathfrak{g}}_{\mathcal{O}}$.
Note that $L_x$ is an $\mathcal{O}$-lattice in $\tilde{\mathfrak{g}}$. By the definition of the height, for all $\mathbf u\in L_x$ we have $\|\mathbf u\|\geq \mathrm{ht}(x)^{-1}$. 

Since the  covolume of $L_x$ is independent of $x$ (and equals the covolume of $\tilde{\frakg}_{\mathcal{O}}$), by lattice reduction theory, one can find a basis $\bu_1, \dots, \bu_{d}$ (where $d\leq 2N^2$) of $L_x$ such that $\| \bu_i\| \ll \mathrm{ht}(x)^{d}$, see \cite[Thm.~ 1.2]{S-adic-Minkowski} and references therein.  

Thus, by choosing the constants $m,\ref{k: inj rad}$ sufficiently large, we get
$$ \| \disg_{2,\tw}^{-1}\disg_{1,\tw} \bu_i-\bu_i\| < \mathrm{ht}(x)^{-1} \mbox{ for all $1 \leq i \leq d$}.$$
In that case, the vector $\disg_{2,\tw}^{-1}\disg_{1,\tw}$ fixing the lattice $L_x$ setwise implies that it is, in fact, fixing $L_x$ pointwise. Thus, $\disg_{2,\tw}^{-1}\disg_{1,\tw}$ belongs to the center of $\tGw$. Let $g'\in G'$ be so that $x=g'\G'(F)$. 
Then, we have $\disg_2^{-1}\disg_1=g'\gamma g'^{-1}$ where $\gamma=(\gamma^1,\gamma^2)\in\G'(F)$ and $\gamma^i$ embeds diagonally in $\G(\adele)$. Now since $\disg_{2,w}^{-1}\disg_{1,w}$ is central, we conclude that $(\gamma^1_w,\gamma^2_w)$ is central, which is impossible if $m$ is big enough.
\end{proof}

\begin{proposition}
    [Non-divergence estimate]\label{prop;non-div} 
    There are positive constants $\constk\label{k:non-div norm bound}$ and $\constk\label{k:non-divergence expon}$, depending on $N$ and $[F:\tilde F]$, so that for any  $g\in G'$ we have \[
    \mu_{g}\left(X\setminus\Siegel( R)\right)\ll p_w^{\ref{k:non-div norm bound}}R^{-\ref{k:non-divergence expon}},\]
    recall that $p_w={\rm char}(k_w)$. 
\end{proposition}

\begin{proof}
    The proof is similar to that of \cite[Lemma 3.6.1]{EMV} (see \cite[App.~B]{EMV}) where one replaces the use of the non-divergence result by Kleinbock and Margulis~\cite{KM-Nondiv} with Ghosh~\cite{Ghosh}. We will go over the main steps of the proof presented in \cite[App.~B]{EMV} to explicate the differences.
	
	Since $\mathcal O$ is a PID, discrete $\mathcal O$-submodules of $\tilde{\frakg}$ are free, see~\cite[Lemma 4.1]{Ghosh}. Following \cite[\S 4]{Ghosh}, for a discrete $\mathcal O$-module $\Delta=\oplus_{i=1}^\ell\mathcal O{\bf u}_i$, we denote $\norm\Delta=\|{\bf u}_1\wedge\cdots\wedge{\bf u}_\ell\|_v$. Then, $\norm{\Delta}$ is a ``norm like" function in the language of~\cite{Ghosh} that plays the role of the covolume of $\Delta$ in the proof (see the proof of Lemma 4.3 and properties N1-N3 in \cite{Ghosh}). We refer to $\Delta$ as an $\mathcal O$-arithmetic lattice in \[V=\oplus_{i=1}^\ell \tF_{\tilde w}{\bf u}_i\subset\tilde{\frakg}.\]
	
	Let $h\in \tGw$. 
	A subspace $V\subset \tilde{\gfrak}$ is called \emph{$h$-rational} if $V\cap \Ad_h\tilde{\gfrak}_{\mathcal{O}}$ is an $\mathcal O$-arithmetic lattice in $V$;
	we let $\norm{V}_h=\norm{V\cap \Ad_h\tilde{\gfrak}_{\mathcal O}}$. The first step is to show that for any $h$ there are no $h$-rational subspaces which are $H_w$ invariant of low norm.

Let $U=\{u(t)\}$ be a one parameter $\tilde{F}_{\tw}$-unipotent subgroup of $H_w$, and use $\Ad_{\tilde\rho(u(t))}$ in the definition of $h(t)$ in the Proof of \cite[Lemma 3.6.1]{EMV}, i.e.\ $h(t)=\Ad_{\tilde\rho(u(t))}\mid_V\in\sl(V)$, where $V$ is an $h$-rational space for some $h\in \tGw$. 
	
	Now, one can follow the arguments in \cite[Proof of Lemma 3.6.1]{EMV} using the above notations, Lemma \ref{lem: inj radius}, \cite[Thm.\ 5.2]{Ghosh}, and \cite[Property N2]{Ghosh} to show that there exist positive constants $c,\ref{k:non-div norm bound}$ such that for any  $x\in\pi_v(Hg\Gamma)$ 
	\begin{equation}
	\label{e;no-thin-subspace}\mbox{there is no $x$-rational, $H_w$-invariant proper subspace of norm $\leq c p_{\tilde w}^{-\ref{k:non-div norm bound}}$,} 
	\end{equation} 
	where $c$ is an absolute constant and $\ref{k:non-div norm bound}$ depends only on $\dim\bG$.  
    
	Put~$\xi=c p_{\tilde w}^{-\ref{k:non-div norm bound}}$. Since the number of $x$-rational proper subspaces
	of norm at most $\xi$ is finite and $U\subset H_w$, a.e.\ $h\in H_w$ has the property that $hUh^{-1}$ does not leave invariant any proper~$x$-rational subspace of norm $\leq \xi$. 
	Alternatively, we may also conclude for a.e.~$h\in H_w\cap K'_v$
	that~$U$ does not leave any proper~$hx$-rational subspace
	of norm~$\leq\xi$ invariant. 
	
	Since $\mu_g$ is $H$-invariant and $H$-ergodic, by Mautner phenomenon $\mu_g$ is also $U$-ergodic. 
	This also implies that the~$U$-orbit of $hx$ equidistributes with
	respect to $\mu_g$ for a.e.~$h$. We choose~$h\in H_w\cap K_v$
	so that both of the above properties hold for~$x'=hx$.
	
	Let $h'$ be such that $x'=h'\Gamma_{\tw}$. Then, for any $h'\Gamma_{\tw}$-rational subspace $V$, if we let \[
	\psi_V(t):=\norm{\Ad_{u(t)}^{-1}V}_{u(t)^{-1}h'}=\norm{\Ad_{u(t)}^{-1}(V\cap \Ad_{h'}\tilde{\gfrak}_{\mathcal O})},\]
	then either $\psi_V$ is unbounded or equals a constant $\geq \xi.$  
	Thus, by~\cite[Thm 4.4]{Ghosh} there exists a positive constant $\constk\label{non-div-proof-2}$ 
	so that
	\begin{equation*}\label{e;non-div}
	\mu_g(\{t:|t|_{\tilde w}\leq r,x' u(t)\notin\Siegel(\epsilon^{-1})\})
	\ll p_{\tilde w}^{\ref{non-div-proof-2}}(\tfrac{\epsilon}{\xi})^{\alpha} \mu_g(\{t:|t|_{\tilde w}\leq r\}),
	\end{equation*}
	for all large enough $r$ and $\epsilon>0$,  where~$\alpha=\ref{k:non-divergence expon}$ only depends on the degree of the polynomials appearing in the matrix entries for the elements of the one-parameter unipotent subgroup~$U$.
	The lemma now follows as the~$U$-orbit equidistributes with respect to~$\mu_g$.
\end{proof}

\section{Invariant subgroups for the measure \texorpdfstring{$\mu$}m}

Recall that $\mu$ is an $H_g$-invariant measure on $Y_g=g\textbf{H}(\bba)/\textbf{H}(F)$. In this section, we will use tools from homogeneous dynamics combined with algebraic properties of good places which have been established in previous sections to prove Theorem~\ref{thm:main}. When $X$ is not compact, the proof will also use the non-divergence result proved in Lemma~\ref{prop;non-div}.

\subsection{Adelic Sobolev norms} \label{sec:sob-intro}

Let $C^{\infty}(X)$ denote the space of functions which are invariant by a compact open subgroup of $\G(\adele)$. We follow \cite[App.~A]{EMMV} to define Sobolev norms on $C^\infty(X)$, and present some of its properties. 

First, we define a system of projections $\operatorname{pr}_v[m]$ for any unitary $\G(F_v)$-representation for each place $v$ such that $\sum_{m\geq 0}\operatorname{pr}_v[m]=1$.
For any $v\in\Sigma$, let $\operatorname{Av}_v[m]$ be the projection on $K_v[m]$-invariant vectors --- this is simply given by averaging over the compact open subgroup $K_v[m]$ equipped with the probability Haar measure. Put $\operatorname{pr}_v[0]:=\operatorname{Av}_v[0]$ and $\operatorname{pr}_v[m]:= \operatorname{Av}_v[m] - \operatorname{Av}_v[m-1]$ for $m\geq1$.

Let $\mathcal{M}$ be the set of functions mapping $\Sigma$ to the non-negative integers, such that almost every $v$ is mapped to zero.
For any $\underline{m}\in\mathcal{M}$ let \[ \|\underline{m}\|:= \prod_{v} q_v^{m_v},\text{ and}\quad\operatorname{pr}[\underline{m}] := \prod_{v} \operatorname{pr}_v[m_v].\]
Then $\operatorname{pr}[\underline{m}]$ acts on any unitary $\G(\adele)$-representation, 
and $\sum_{\m}\operatorname{pr}[\m]f=f$ for any $f\in C^\infty(X)$. Moreover, for any $f\in C^\infty(X)$ and $\underline{m}\in \mathcal{M}$, the function  $\operatorname{pr}[\underline{m}]f$ is invariant under $K_{\underline{m}}:=\prod_{v}K_v[m_v]$.

Given an integer $d\geq0$ we define a degree~$d$ Sobolev norm by
\begin{equation} \label{Sobnormdef} 
\Sob_d(f)^2 := \sum_{\m} 
 \|\m\|^d\|\operatorname{pr}[\m](1+\height(x))^{d} f(x)\|_2^2,
\end{equation}
where $\operatorname{ht}(\cdot)$ is as defined in \S\ref{s;nondivergence}. 
For a compactly supported smooth function on $X$ any of these Sobolev norms is finite.
It is also easy to see that
\begin{equation}
    \text{$\Sob_d(f)\leq\Sob_{d'}(f)\quad$ if $d<d'.$}
\end{equation}

We claim that in a similar way to \cite[App.~A]{EMMV}, we have the following:

\begin{enumerate}
\item [S0.] For any $d$, $\Sob_d$ is a pre-Hilbert norm on $C_c^\infty(X)$.

\item[S1.] There exists $d_0$, depending on~$N$, and~$[F:\tilde F]$ (see \S\ref{s;nondivergence} for the definition of~$\tilde F$) such that for all $d\geq d_0$, we have $\|f\|_{L^{\infty}}\ll_d\Sob_d(f)$.

\item[S2.] Given $d_0,$ there are $d>d'>d_0$ and an orthonormal basis $\{e_k\}$ of the completion of $C_c^\infty(X)$ with respect to $\Sob_{d}$ which is orthogonal with respect to $\Sob_{d'}$ such that \[
\mbox{$\sum_k\Sob_{d'}(e_k)^2<\infty\quad$ and $\quad\sum_k\tfrac{\Sob_{d_0}(e_k)^2}{\Sob_{d'}(e_k)^2}<\infty.$} \]

\item[S3.] 
Let $g\in \G(\adele)$ and $f\in C_c^\infty(X)$. 
We write~$g.f$ for the  action of~$g$ on~$f$, and $\|g\|=\prod_{v\in\Sigma}\|g_v\|$. Then, for every $d\geq 0$ 
\[
\Sob_d(g.f)\ll \|g\|^{4Nd}\Sob_d(f),
\]
where the implied constant is absolute.
If in addition~$g\in K$, then we have~$\Sob_d(g.f)=\Sob_d(f)$.
\end{enumerate}

\medskip
Now let us fix a good place~$w$. Note that for every $\lambda\in\Phi$, we have $\|u_\lambda'(s)\|\leq(1+|s|_w)^N$ for all $s\in F_w.$

\begin{enumerate}
\item[S4.] 
For any $d\ge d_0$, any $r\geq0$, any $g\in K_w[r]$, and any $f\in C_c^\infty(X)$, we have 
\[
\|g.f-f\|_\infty\leq q_w^{-r}\Sob_d(f).
\]

\item[S5.] There exist $d_0,M$, which depend only on~$N$, so that the following holds. For every $\levelbound$, let $\operatorname{Av}_L$ be the operation of averaging over $K_w[\levelbound]$.
Let $\lambda\in\Phi$ and $s\in F_w$. Put $\mathbb{T}_s=\operatorname{Av}_\levelbound \star\, \delta_{u_\lambda'(s)} \star \operatorname{Av}_\levelbound$ where $\star$ denotes convolution of operators. 
For all $x\in X$, all $f\in C_c^\infty(X)$, and~$d\geq d_0$ we have
\[
\Bigl|\mathbb{T}_sf(x)-\int f\operatorname{d}\!m_X\Bigr|\ll q_w^{(d+2)\levelbound}\height(x)^d \|\mathbb{T}_s\|_{2,0}\Sob_d( f),
\]
where $\|\mathbb{T}_s\|_{2,0}$ denotes the operator norm of $\mathbb{T}_s$ on $L^2_0(X,m_X)$. 
Moreover, we have 
\begin{equation}\label{eq: Ts 2-norm}
    \|\mathbb T_s\|_{2,0}\ll (1+|s|_w)^{-1/2M} q_w^{2 d \levelbound}
\end{equation}

\item[S6.] There exist $d_0,M$, which depend only on~$N$, so that for all $d\geq d_0$ and all $\lambda\in\Phi$, we have 
\begin{equation}\label{eq:sobnormmc}
\Big| \langle u_\lambda(s) f_1, f_2 \rangle_{L^2(\mu)} - 
	\int f_1\operatorname{d}\!\mu \int\bar f_2\operatorname{d}\!\mu \Big|\\ 
 \ll (1+|s|_w)^{-1/2M} \Sob_{d}(f_1) \Sob_d(f_2).
\end{equation}
\end{enumerate}

Indeed, the above properties can be shown arguing in a similar way to \cite[App.~A4]{EMMV}. We only present a brief guided tour of the proof, and the implied changes in the statements. In the proof of property S1., Lemma \ref{lem: inj radius} replaces \cite[Lemma 7.2]{EMMV}, and so the implied constant only depends on $N$ and $[F:\tilde F]$. In the proof of property S3., for any place $v$ and $g\in G_v$, the inequality \[
\operatorname{ht}(gx)\ll\norm{g}^{2N}\operatorname{ht}(x)\] 
follows from the definition of $\operatorname{ht}(x)$ (see \S\ref{s;nondivergence}), and so (except possibly a different power in the result) the proof is the same. Properties S5.\ and S6.\ follow from $1/M$-temperedness (see \cite[Thm.~2]{CHH1988}) of the natural representation of $\G(F_w)$ on $L^2_0(X,m_X)$ and $L^2_0(X,\mu)$, respectively. In the case where $\G$ is type $A_1$, this follows from property $(\tau)$,~\cite{Drinfeld}. In all other cases, $\G$ has absolute rank at least 2, and $w$ is so that $\G$ is $F_w$-split (see \S\ref{sec: good place}). Therefore, $\G(F_w)$ has property $(T)$ and the desired $1/M$-temperedness follows from~\cite{HeeOh-T}.

\subsection{Generic points}\label{sec:generic pts} 
Let $\lambda\in\Phi$.  
Let $n\in\bbn$. For any $x\in X$ and $t=\varpi_w^{-n}$, define
\begin{equation}\label{def: averaging operator} D_nf(x):=q_w^{n}\int_{B(0,q_w^{-n})}f(a_\lambda'(t) u_\lambda'(s)x)\diff\!s-\int_Xf\diff\!\mu;
\end{equation}
see the beginning of~\S\ref{sec:effective generation}, more specifically~\eqref{eq: definition of A'}, for the definition of $a'_\lambda(t)$ and $u'_\lambda(s)$. In particular, we have $a'_\lambda(t), u'_{\lambda}(s)\in H_{g,w}\subset H_g$ for all $t,s,\lambda$. Therefore, they preserve the measure $\mu$. 

A point $x\in X$ is called $n_0$-{\em generic} for $\lambda$ with respect to the Sobolev norm $\scal_d$, if for any integer $\ell\geq n_0$, we have   
\begin{equation}\label{generic}
|D_\ell f(x)|\leq q_w^{-\ell/8M}\scal_d(f)
\end{equation}
where $M$ is as in~\eqref{eq:sobnormmc}.

\begin{proposition}\label{prop:generic set}
If $d_0$ is chosen large enough, depending on $N$, then for all $d\geq d_0$ the $\mu$-measure of points which are not $n_0$-generic for some $\lambda\in\Phi$ with respect to $\scal_d$ is $\ll q_w^{-n_0/5M}$. 
\end{proposition}

\begin{proof}
Let $\lambda\in\Phi$ be arbitrary.  
For simplicity in the notation, we will write $a'(t)$ and $u'(s)$ for $a'_\lambda(t)$ and $u'_\lambda(s)$, respectively. 

Let $\ell\geq n_0$ be an integer and $d_0$ satisfy~\eqref{eq:sobnormmc}.
Let $f$ be a fixed function in $C_c^\infty(X)$. 
Let us write $B=B(0,q_w^{-\ell})$. Put
\[
A_1=\{(s_1,s_2)\in B\times B: |s_1-s_2|_w<q_w^{-11\ell/10}\}
\]
and $A_2=B\times B\setminus A_1$. 

For $(s_1,s_2)\in B\times B$, put $F(s_1,s_2)=\langle a'(t)u'(s_1-s_2)a'(t^{-1})f, f\rangle_{\mu}$. If $(s_1,s_2)\in A_2$ and $|t|_w\geq q_w^{\ell}$, then $a'(t)u'(s_1-s_2)a'(t^{-1})=u'(\tau)$ with $|\tau|\geq q_w^{0.9\ell}$. Hence, for any $(s_1,s_2)\in A_2$ and $|t|_w\ge q_w^\ell$, by~\eqref{eq:sobnormmc} we have 
\[
\Bigl|F(s_1,s_2)-\Bigl(\int_Xf\diff\!\mu\Bigr)^2\Bigr|\ll q_w^{-9\ell/20M}\Sob_{d_0}(f)^2.
\]

Since the Lebesgue measure of $A_1$ is $q^{-21\ell/10}$, using the definition of $D_\ell f(x)$ and the Fubini's theorem, now we have  
\begin{align*}
\int_X|D_\ell f(x)|^2\diff\!\mu&=q_w^{2\ell}\int_B\int_BF(s_1,s_2) \diff\!s_1\diff\!s_2-\Bigl(\int_Xf\diff\!\mu\Bigr)^2 \\ 
&=q_w^{2\ell}\Bigl(\int_{A_1}F(s_1,s_2)\diff\!s_1\diff\!s_2+\int_{A_2}F(s_1,s_2)\diff\!s_1\diff\!s_2\Bigr)-\Bigl(\int_Xf\diff\!\mu\Bigr)^2\\
&\ll 3q_w^{-9\ell/20M}\Sob_{d_0}(f)^2
\end{align*}
Using Markov's inequality we arrive at
\begin{equation}\label{eq:chebechev}
\mu(\{x\in X:|D_\ell f(x)|>\vare\})\ll 3{\vare^{-2}q_w^{-9\ell/20M}\scal_{d_0}(f)^2}.
\end{equation}

To conclude, we use property S2.\ of the Sobolev norms. That is, there are $d>d'>d_0$ and an orthonormal basis $\{e_k\}$ of the completion of $C_c^\infty(X)$ with respect to $\Sob_{d}$ which is orthogonal with respect to $\Sob_{d'}$ so that 
\begin{equation}\label{eq:reltr-cont}
\mbox{$\sum_k\Sob_{d'}(e_k)^2<\infty\quad$ and $\quad\sum_k\tfrac{\Sob_{d_0}(e_k)^2}{\Sob_{d'}(e_k)^2}<\infty$.} 
\end{equation}
Put $c=(\sum_k\Sob_{d'}(e_k)^2)^{-1/2}$ and define 
\begin{equation}\label{eq:genericset}
X'=\bigcup_{\ell\geq n_0, i\geq 1}\Bigl\{x\in X: |D_\ell{e_i}(x)|\geq c q_w^{-\ell/8M}\scal_{d'}(e_i)\Bigr\}.
\end{equation} 
Using~\eqref{eq:chebechev} and~\eqref{eq:reltr-cont}, we have 
\begin{equation}\label{eq:genericmeasure}
\mu(X')\ll q_w^{-n_0/5M}.
\end{equation} 

Let $f \in C_c^\infty(X)$, and write $f=\sum f_k e_k$ and  suppose $x \not\in X'$.  
Let $\ell\geq n_0$, then using the triangle inequality for $D_\ell$ we obtain
\begin{align*} 
|D_\ell f(x)| &\leq \sum_k |f_k||D_\ell e_k(x)|\leq c q_w^{-\ell/8M} \sum_k |f_k| \Sob_{d'}(e_k)\\
&\leq cq_w^{-\ell/8M} \Bigl(\sum_{k} |f_k|^2\Bigr)^{1/2} \Bigl(\sum_k \Sob_{d'}(e_k)^2 \Bigr)^{1/2}
\\ &= q_w^{-\ell/8M} \Sob_{d}(f),
\end{align*}
where  the last inequality follows from the definition of $c$ and since $\{e_i\}$ is an orthonormal basis with respect to $\Sob_d$.

This completes the proof in view of~\eqref{eq:genericmeasure} as $\lambda$ was arbitrary. 
\end{proof}

For the rest of this section, we fix some $d\geq d_0$ so that Proposition~\ref{prop:generic set} holds, and write $\Sob=\Sob_d$.

For every $m$, let $Q_m=\bigl\{((g_v),1)\in K: g_w\in \pi_w^1(K_w[m])\bigr\}$.

\begin{proposition}\label{prop:almost inv1}
There exists $\constk\label{k:almost inv1}$ such that for all $c\in\bbn$ and all large enough $m\geq 3c$, the following holds. 
Let $\disg\in Q_m$ satisfy $\disg_w\neq 1$ and 
\[
\|\disg_{w,\lambda}-I\|\geq q_w^{-c}\|\disg_w-I\|.
\]
Let $\lambda\in \Phi$. Suppose $y_1, y_2\in X$ are two $\lfloor m/10\rfloor$-generic points for the measure $\mu$ with respect to $\Sob$ satisfying that $y_2=\disg y_1$. Then,  
\[
|u_\lambda(\tau).\mu(f)-\mu(f)|\ll q_w^{-m\ref{k:almost inv1}}\Sob(f)
\]
for all $\tau\in B(0,1)$ and all $f\in C_c^\infty(X)$. 
\end{proposition}

\begin{proof}
As in the proof of Proposition \ref{prop:generic set}, we write $a'(t),u'(s)$ for $a'_\lambda(t),u'_\lambda(s)$; we also write $a(t),u(s)$ for $a_\lambda(t),u_\lambda(s)$. 
Recall from~\eqref{eq: a and a' conjugation} that
\[
a(t) \disg a(t^{-1})= a'(t)\disg a'(t^{-1})\;\;\text{and}\;\; u(s) \disg u(-s)= u'(s)\disg u'(-s).
\]

Let us also recall that $a'(t)$ and $u'(s)$ belong to $H_g$, hence, they leave $\mu$ invariant; while 
$a(t), u(s)\in G_w$ and we do not have any a priori information regarding there action on $\mu$.

For all $s$ and $t$, put
\[
\disg_{s,t}=a'(t)u'(s)\disg u'(-s)a'(t^{-1});
\]
note that $\disg_{s,t,v}=\disg_{v}$ for all $v\neq w$. 

Apply Lemma~\ref{lem:sl2-rep} with $m_1=c$, and let $m\geq 3c$.
Let us write $\ell=\lfloor (m_{\disg_w}-c)/2\rfloor$ where 
$\disg_w\in K_w[m_{\disg_w}]\setminus K_w[m_{\disg_w}+1]$.
Let $|t|_w=q_w^{\ell}$; we will always assume $|s|_w\leq |t|_w^{-1}$.
By Lemma~\ref{lem:sl2-rep}, we have 
\begin{equation}\label{eq:conjugate lambda}
\Bigl\|\disg_{s,t,w}-a(t)\disg_{\lambda,w} a(t^{-1})\|\leq \ref{C:sl2-rep} q_w^{-\ell+c};
    \end{equation}
moreover, if $|t|_w=q_w^\ell\geq \ref{C:sl2-rep}q_w^{2c+3}$, then  
\begin{equation}
    \label{eq:conjugate big}\|\disg_{s,t,w}-I\|\geq q_w^{-c-2}.
\end{equation}

Recall now that $y_i$, $i=1,2$, is $\lfloor m/10\rfloor$-generic for $\mu$ w.r.t.\ $\Sob$. 
This, together with our choice of the integer $\ell\geq \lfloor m/10\rfloor$, implies that
\begin{equation}\label{eq:extra inv}
\Bigl|q_w^{\ell}\int_{B_w(0,q_w^{-\ell})}f(a'(t) u'(s)y_i)\diff\!s-\int_Xf\diff\!\mu\Bigr|\leq q_w^{-\ell/8M}\Sob(f)
\end{equation}
for $i=1,2$. 

Since $y_2=\disg y_1$, we get that $a'(t)u'(s)y_2=\disg_{s,t}a'(t)u'(s)y_1$. In particular, 
\[
f(a'(t) u'(s)y_2)=f(\disg_{s,t}a'(t) u'(s)y_1).
\]
Applying~\eqref{eq:extra inv} with $y_2$, we thus conclude that  
\begin{equation}\label{eq:conjugate lambda -1}
\Bigl|q_w^{\ell}\int_{B_w(0,q_w^{-\ell})}f(\disg_{s,t}a'(t) u'(s)y_1)\diff\!s-\int_Xf\diff\!\mu\Bigr|\leq q_w^{-\ell/8M}\Sob(f).
\end{equation}

Put $\hat\disg=a(t)\disg'a(t^{-1})$ where $\disg'_{v}=\disg_v$ for all $v\neq w$, and $\disg_{w}'=\disg_{w,\lambda}$. 
Then~\eqref{eq:conjugate lambda -1} and~\eqref{eq:conjugate lambda} imply that if $m\geq 5c$, then 
\[
\Bigl|q_w^{\ell}\int_{B_w(0,q_w^{-\ell})}f\Bigl(\hat\disg a'(t) u'(s)y_1\Bigr)\diff\!s-\int_Xf\diff\!\mu\Bigr|\ll q_w^{-\ell/9M}\Sob(f).
\]
The above,~\eqref{eq:extra inv}, applied with $y_1$ and $\hat\disg^{-1}.f$, and property S3.\ yield
\begin{equation}\label{eq: almost inv2}
|\hat\disg.\mu(f)-\mu(f)|\ll  q_w^{-\ell/9M}\Sob(f).
\end{equation}

Recall that $\|\disg_{w,\lambda}-I\|\geq q_w^{-c}\|\disg_w-1\|$. Therefore, if we write $u(\tau_0)=a(t)\disg_{w,\lambda} a(t^{-1})$, then $q_w^{-c-2}\leq |\tau_0|_w\leq 1$. 

For every $\rho$, put $\hat\disg_\rho:=a'(\rho)\hat\disg a'(\rho^{-1})$.
Then $\hat\disg_{v,\rho}=\hat\disg_v$ and $\hat\disg_{w,\rho}=u(\rho^2\tau_0)$. 
Moreover, we have 
\begin{equation}\label{eq:iteration}
\begin{aligned}
|\hat\disg_\rho.\mu(f)-\mu(f)|&=|\hat\disg.\mu(a'(\rho^{-1}).f)-\mu(a'(\rho^{-1}).f)|\\
&\ll q_w^{-\ell/9M}\scal(a'(\rho^{-1}).f). 
\end{aligned}
\end{equation}
where we use the fact that $\mu$ is $H_g$-invariant in the first equality and~\eqref{eq: almost inv2} in the second line.  
In view of property S3., we have $\scal(a'(\rho^{-1})f)\ll \max\{|\rho|_w^{4Nd},|\rho|_w^{-4Nd}\}\scal(f)$. This and~\eqref{eq:iteration}, imply 
\[
\Bigl|\hat\disg_\rho.\mu(f)-\mu(f)\Bigr|\ll q_w^{-\ell/9M}\max\{|\rho|_w^{4Nd}, |\rho|_w^{-4Nd}\}\scal(f).
\]
Recall now that $\hat\disg^{-1}\hat\disg_{\rho}=u((\rho^2-1)\tau_0)$. Altogether we cocnlude that    
\[
\bigl|u((\rho^2-1)\tau_0).\mu(f)-\mu(f)\bigr|\ll q_w^{-\ell/9M}\scal(f).
\]
for all $|\rho|_w=1$. We now consider two cases:

\medskip

{\em Case 1}: ${\rm char}(F)>2$. Then by Lemma~\ref{lem:implicit function}, $B(0,\ref{k:implicit func})\subset \{\rho^2-1: |\rho|_w=1\}$. This and the fact that $|\tau_0|_w\geq q_w^{-c-2}$, imply that 
\[
\bigl|u(\tau).\mu(f)-\mu(f)\bigr|\ll q_w^{-\ell/9M}\scal(f)
\]
for all $\tau\in B(0, q_w^{-c-2}\ref{k:implicit func})$. 
Using $u(\rho^2\tau)=a'(\rho)u(\tau)a'(\rho^{-1})$ and arguing as in~\eqref{eq:iteration} again, we conclude that so long as $m$ (and hence $\ell$) is large enough, 
\[
|u(\tau).\mu(f)-\mu(f)|\ll q_w^{-\ell/10M}\scal(f)
\]
for all $\tau\in B(0,1)$; the proof is complete in this case.  

\medskip

{\em Case 2}: ${\rm char}(F)=2$. Then by Lemma~\ref{lem:implicit function}, we have 
\[
\bigl|u(\tau).\mu(f)-\mu(f)\bigr|\ll q_w^{-\ell/9M}\scal(f)
\]
for all $\tau=\tau_0 \alpha^2$ and $\alpha\in B(0,1)$. Using $u(\rho^2\tau)=a'(\rho)u(\tau)a'(\rho^{-1})$ and arguing as in~\eqref{eq:iteration} again, we conclude that so long as $m$ is large enough, 
\begin{equation}\label{eq: alm inv char 2}
    |u(\tau).\mu(f)-\mu(f)|\ll q_w^{-\ell/10M}\scal(f)
\end{equation}
for all $\tau=\tau_0\alpha^2$ with $|\tau|\leq q_w^{c+2}$ (recall that $q_w^{-c-2}\leq |\tau_0|\leq 1$). 

Now by Lemma~\ref{lem:char 2}, for every $r\in B(0,1)$ 
\[
u_\lambda(r)= h_1\cdots h_{\ell'}
\]
where $\ell'$ is absolute and $h_i\in H_{g,w}$ with $\|h_i\|\leq 1$ or $h_i\in \{u_{\lambda}(\tau_0\alpha^2): |\tau_0\alpha^2|\leq |\tau_0|^{-1}\}$. This and~\eqref{eq: alm inv char 2} complete the proof in this case as well.
\end{proof}

Recall that $\vol_{G'}$ denotes our fixed Haar measure on $G'$. 
Also recall that for every $m$, we let 
\[
Q_m'=\Bigl\{\bigl((g_v^1),(g^2_v)\bigr)\in K': (g_w^1, g_w^2)\in \pi_w(K_w'[m])\Bigr\}.
\]

\begin{lemma}\label{lem:pigeonhole}
Assume $n$ is so that  
\[
\vol_{G'}(Q_n')^{-1}\leq \frac{\vol(Y)}{4(\# {\bf C}(F))}\vol(Y),
\]
where ${\bf C}$ denotes the center of $\G'$. 
There exist points $y_1,y_2$ and $\lambda\in\Phi$ so that 
\begin{enumerate}
    \item $y_1$ and $y_2$ are $\lfloor n/20\rfloor$-generic for $\lambda$ and the measure $\mu$ w.r.t.\ $\Sob$.  
    \item $y_2=\disg y_1$ where $\disg\in Q_{n-m_0}\setminus {\rm Stab}(\mu)$. Moreover, $\disg_w\neq 1$ and we have  
    \[
    \|\disg_{w,\lambda}-I\|\geq q_w^{-m_0}\|\disg_w-I\|,
    \]
    where $m_0$ is as in Lemma~\ref{lem:Chevalley irred rep}. 
\end{enumerate}
\end{lemma}

\begin{proof}
We may assume $\vol(Y)$ is large enough so that $n$ maybe chosen with $\lfloor n/20\rfloor\geq 2m_0+1$. 
Let 
\[
Y_{\rm good}'(R)=\mathfrak{S}(R)\cap \{y\in Y: \text{ $y$ is $\lfloor n/20\rfloor$-generic}\}.
\]
Then by Propositions~\ref{prop:generic set} and~\ref{prop;non-div}, if $R\gg_{q_w} 1$ is large enough, we have 
    \begin{equation}\label{eq:Ygood is large}
    \mu(Y_{\rm good}'(R))\geq 1-q_w^{-11m_0\dim\G}.
    \end{equation}
Fix one such $R$ for the rest of the argument, and put $Y'_{\rm good}=Y'_{\rm good}(R)$.

Let $\mathsf H=\Delta_g(K_wA_{-m_0}K_w)$.
Put 
\[
Y_{\mathsf H}=\{(h,y)\in \mathsf H\times Y'_{\rm good}: hy\in Y'_{\rm good}\}. 
\]
Note that for every $h\in \mathsf H$ and all $y\in Y'_{\rm good}\cap h^{-1}Y'_{\rm good}$, we have $(h,y)\in Y_{\mathsf H}$. Since $\mathsf H\subset H_{g,w}$ leaves $\mu$ invariant, it follows from this observation and~\eqref{eq:Ygood is large} that for every $h\in\mathsf H$, the fiber $\{y: (h,y)\in Y_{\mathsf H}\}$ has measure $\geq 1-2q_w^{-11m_0\dim\G}$. Therefore, by Fubini's theorem, there exists a subset $Y_{\rm good}\subset Y'_{\rm good}$ with 
\[
\mu(Y_{\rm good})\geq 1-q_w^{-5m_0\dim\G}
\]
so that if for every $y\in Y_{\rm good}$ we put $\mathsf H_y=\{h\in \mathsf H: hy\in Y'_{\rm good}\}$, then 
\[
|\mathsf H\setminus \mathsf H_y|< q_w^{-5m_0\dim G}.
\]

Let $\{Q_n'.x_i: i\in I\}$ be a a disjoint covering of $\mathfrak{S}(R)$, then $\#I\leq \vol_{G'}(Q_n')^{-1}$, moreover, since $\mu(Y_{\rm good})\geq 0.9$, there exists some $i_0$ so that 
\[
\mu(Q_n'.x_{i_0}\cap Y_{\rm good})\geq \frac{1}{2\cdot (\# I)}. 
\]

    Recall now from Lemma~\ref{lem:stabilizer}, that ${\rm Stab}(\mu)=H\cdot {\bf C}(F)$.
    We now claim that there exist $y_1', y_2'\in Q_n'.x_{i_0}\cap Y_{\rm good}$ so that $y_2'=\disg'y_1'$ for some $\disg'\not\in {\rm Stab}(\mu)$.  Assume to the contrary that for every $y,y'\in Q_n'.x_{i_0}\cap Y_{\rm good}$, we have $y'=hy$ where $h\in {\rm Stab}(\mu)$. 
    Then, since $b\mapsto b x$ is injective for all $x\in\mathfrak S(R)$ and all $b\in Q_n'$, see Lemma~\ref{lem: inj radius}, we have 
    \[
    \mu(Q_n'.x_{i_0}\cap Y_{\rm good})\leq (\# {\bf C}(F))\cdot \mathsf{m}(Q_n' \cap H).
    \]
    From this we conclude that
    \begin{align*}
    \vol_{G'}(Q_n')&\leq (\#I)^{-1}\leq 2\mu(Q_n' x_{i_0}\cap Y_{\rm good})\\
    &\leq 2\cdot (\# {\bf C}(F))\cdot \mathsf{m}(Q_n' \cap H)\leq 2\cdot (\# {\bf C}(F)) \cdot \vol(Y)^{-1}
    \end{align*}
    which contradicts our assumption and proves the claim. 
    
    We now use Lemma~\ref{lem:adjustment-1} to move $\disg'$ to $\disg\in Q_{m-2m_0}$ while keeping points generic; moreover, 
    we need $\disg_w\neq 1$ and that $\|\disg_{w,\lambda}-I\|\gg\|\disg_w-I\|$ for some $\lambda\in\Phi$. 
   
Let $E_1,E_2\subset K_wA_{-m_0}K_w$ be so that 
\[
\Delta_g(E_i)=\mathsf H_{y_i'}\quad\text{and}\quad|K_wA_{-m_0}K_w\setminus E_i|\leq q_w^{-5m_0\dim\G}. 
\]
For $i=1,2$, let $\alpha_i\in \mathsf H_{y_i'}$ and $\lambda\in\Phi$ 
satisfy the conclusion of Lemma \ref{lem:adjustment-1}, and put $\disg:=\alpha_2 \disg'\alpha_1^{-1}$.
Note that $\disg\in Q_{n-2m_0}$, since $\disg_v=\disg'_v$ for any $v\ne w$, and by Lemma \ref{lem:adjustment-1} we have 
\[
\alpha_2 \disg_w'\alpha_1^{-1}=\disg_w\in K_w[n-2m_0]\subset G_w\times \{1\}
\]
satisfies $\|\disg_{w,\lambda}-I\|\gg \|\disg_w-I\|$. Moreover, since $\alpha_i\in {\rm Stab}(\mu)$ and $\disg'\not\in{\rm Stab}(\mu)$, we have $\disg\not\in{\rm Stab}(\mu)$.  

Let $y_i=\alpha_iy'_i$. Since $\alpha_i\in\mathsf H_{y'_i}$, we have $y_i\in Y'_{\rm good}$. 
Furthermore, since $y'_2=\disg' y'_1$, 
\[
y_2=\alpha_2y'_2= \alpha_2 \disg'\alpha_1^{-1} \alpha_1 y_1'=\disg y_1.
\]

It only remains to show that $\disg_w\neq 1$. Assume to the contrary that $\disg_w=1$. Then for every $\ell\geq \lfloor n/20\rfloor$ and $i=1,2$, we have 
\[
\Bigl|q_w^{\ell}\int_{B_w(0,q_w^{-\ell})}f(a'_\lambda(t) u'_\lambda(s)y_i)\diff\!s-\int_Xf\diff\!\mu\Bigr|\leq q_w^{-\ell/8M}\Sob(f)
\]
for all $f\in C_c^\infty(X)$. 
Now since $\disg_w=1$ and $a'_\lambda(t),u'_\lambda(t)\in H_{g,w}\subset G'_w$, we have 
\[
a'_\lambda(t)u'_\lambda(s)y_2=a'_\lambda(t)u'_\lambda(s)\disg y_1= \disg a'_\lambda(t)u'_\lambda(s) y_1
\]
for all $t,s$. This and the above imply that $\mu(\disg^{-1}.f)=\mu(f)$. That is, $\disg\in{\rm Stab}(\mu)$ which is a contradiction. The proof is complete. 
\end{proof}

\begin{corollary}\label{cor: almost inv Kw}
    There exists some $\constk\label{k: Kw final}$ and $\levelbound$ so that 
    \[
    |\mu(b.f)-\mu(f)|\ll \vol(Y)^{-\ref{k: Kw final}}\Sob(f),
    \]
    for every $b\in K_w'[\levelbound]$ and every $f\in C_c^\infty(X)$. \end{corollary}

\begin{proof}
Since $K'_w[L]=K_w[L]\cdot (H_g\cap K'_w[L])$ for every $L\geq 0$ and $H_g$ leaves $\mu$ invariant, it suffices to find $\ref{k: Kw final}$ and $\levelbound$ so that the claim holds for all $b\in K_w[\levelbound]$.  

Let $n\in\bbn$ be so that 
\[
\vol_{G'}(Q_n')^{-1}\leq \frac{\vol(Y)}{4(\# {\bf C}(F))}\vol(Y),
\] 
see Lemma~\ref{lem:pigeonhole}. Then by that lemma, there exist $y_1$, $y_2=\disg y_1$, and $\lambda\in\Phi$ which are both $\lfloor n/20\rfloor$-generic for $\lambda$ and the measure $\mu$, so that $\disg\in Q_{n-2m_0}$, $\disg_w\neq 1$, and   
    \[
    \|\disg_{w,\lambda}-I\|\geq q_w^{-m_0}\|\disg_w-I\|.
    \]
    
    Let $m=n-2m_0$, Assuming $n$ is large enough, we have $\lfloor m/10\rfloor \geq \lfloor n/20\rfloor$.
    Apply Lemma~\ref{prop:almost inv1} with $c=m_0$ and these $y_1$, $y_2=\disg y_1$, and $m$. We thus conclude that 
    \begin{equation}\label{eq: almost inv use}
|u_\lambda(\tau).\mu(f)-\mu(f)|\ll q_w^{-m\ref{k:almost inv1}}\Sob(f)
\end{equation}
for all $\tau\in B(0,1)$ and all $f\in C_c^\infty(X)$. 

In view of Lemma~\ref{lem:effectively generation}, there exist $h_1,\ldots, h_\ell\in H_g$ so that every $b\in K_w[\ref{k: Gw generation}]$ may be written as 
$b=b_1\cdots b_\ell$,
where $b_i\in h_iU_\lambda [0]h_i^{-1}$ for every $1\leq i\leq \ell$. 

Since $h_i\in H_g$ leave the measure $\mu$ invariant, and we may choose $h_i$ so that $\|h_i\|\leq 1$, the above and~\eqref{eq: almost inv use} imply the claim with $L=\ref{k: Gw generation}$ and $\ref{k: Kw final}=\star\ref{k:almost inv1}$.
\end{proof}

\begin{proof}[Proof of Theorem~\ref{thm:main}]
 The proof goes along the same lines as the proof of~\cite[Thm.~1.5]{EMMV}; 
 we recall the details for the convenience of the reader.
 
 Let $f\in C_c^\infty(X)$. 
 Let $\levelbound$ be as in Corollary~\ref{cor: almost inv Kw}, and let $\operatorname{Av}$ be the operation of averaging over $K_w[\levelbound]$. Let us also fix some $\lambda\in\Phi$ and let $\delta_{u'_\lambda(s)}$ be the delta mass at $u'_{\lambda}(s)$ for some $s\in F_w$. Let $\mathbb T_s=\operatorname{Av} \star \delta_{u'_\lambda(s)} \star \operatorname{Av}$, where $\star$ denotes the convolution operation. 
 In view of Corollary~\ref{cor: almost inv Kw} and the fact that $u'_\lambda(s)\in H_g$ leaves $\mu$ invariant, we conclude that 
 \begin{equation}\label{AI}
|\mu(f)-\mu(\mathbb T_s f)|\leq
\vol(Y)^{-\star} (\Sob(\delta_{u'_\lambda(s)} \star \operatorname{Av} * f)+\Sob(f)).
\end{equation}
Therefore, 
\begin{equation}\label{eq: muf mXf}
\begin{aligned}
 |\mu(f)-m_X(f)|&\leq |\mu(\mathbb T_s f)-m_X(f)|+\vol(Y)^{-\star} (\Sob(\delta_{u'_\lambda(s)} \star \operatorname{Av} * f)+\Sob(f))\\
 &\leq|\mu(\mathbb T_s f)-m_X(f)|+\vol(Y)^{-\star}(1+|s|_w)^{4d}\Sob(f)   
\end{aligned}
\end{equation}
where we used $\Sob(\delta_{u'_\lambda(s)} \star \operatorname{Av} * f\leq (1+|s|_w)^{4Nd}\Sob(f)$.
We now estimate the first term on the right side of the above. 
First recall from S5.\ that 
\[
\Bigl|\mathbb{T}_sf(x)-\int f\operatorname{d}\!m_X\Bigr|\ll q_w^{(d+2)L} \height(x)^d\|\mathbb{T}_s\|_{2,0}\Sob( f),
\]
where $\|\mathbb{T}_s\|_{2,0}$ denotes the operator norm of $\mathbb{T}_s$ on $L^2_0(X,m_X)$. Therefore, 
\begin{align*}
|\mu(\mathbb T_s f)-m_X(f)|&=\Bigl|\int \mathbb T_s f-m_X(f)\operatorname{d}\!\mu\Bigr|\\
&\leq\Bigl|\int_{\mathfrak S(R)}\mathbb T_s f-m_X(f)\operatorname{d}\!\mu\Bigr|+\Bigl|\int_{X\setminus \mathfrak S(R)} \mathbb T_s f-m_X(f)\operatorname{d}\!\mu\Bigr|\\
&\leq q_w^{(d+2)\levelbound} R^d\|\mathbb{T}_s\|_{2,0}\Sob( f) + q_w^{\star}R^{-\star}\|f\|_\infty\\
&\leq \Bigl((1+|s|_w)^{-1/2M} q_w^{(3 d+2) \levelbound}R^d+q_w^\star R^{-\star}\Bigr)\Sob(f),
\end{align*}
in the last inequality we used $\|\mathbb T_s\|_{2,0}\ll (1+|s|_w)^{-1/2M} q_w^{2 d \levelbound}$, see~\eqref{eq: Ts 2-norm}.  

Optimizing $|s|_w$ and $R$, and using the fact that $q_w\ll(\log\vol(Y))^2$, we get the theorem from the above and~\eqref{eq: muf mXf}.
\end{proof}

\bibliographystyle{amsplain}
\bibliography{positive_proptau}

\end{document}